\documentclass[12pt]{article}
\pdfoutput=1

\usepackage{amsmath,amsthm,amsfonts,amssymb,amscd,changepage}
\usepackage{enumitem}
\usepackage[margin=1in]{geometry}
\usepackage[pdftex,colorlinks,citecolor=black,linkcolor=black,urlcolor=black,bookmarks=false]{hyperref}
\usepackage{mathtools}
\usepackage{tikz}
\usepackage{microtype}

\newcommand{\pagea}{FOCS18-285}
\newcommand{\pageb}{FOCS18-315}

\theoremstyle{plain}
\newtheorem{theorem}{Theorem}[section]
\newtheorem{lemma}[theorem]{Lemma}
\newtheorem{corollary}[theorem]{Corollary}

\newtheorem{maintheorem}{Theorem}

\newtheorem{mainlemma}[maintheorem]{Lemma}

\theoremstyle{definition}
\newtheorem{definition}[theorem]{Definition}
\newtheorem{remark}[theorem]{Remark}

\numberwithin{equation}{section}

\newcommand\C{{\mathbb{C}}}

\newcommand\F{{\mathbb{F}}}
\newcommand{\R}{\mathbb{R}}
\newcommand\Z{{\mathbb{Z}}}

\newcommand\Tr{{\mathop\textup{tr}}}
\newcommand\irr{{\mathop\textup{Irr}}}

\newcommand{\GL}{\textup{GL}}
\newcommand{\SL}{\textup{SL}}

\newcommand{\U}{\textup{U}}
\newcommand{\SU}{\textup{SU}}
\newcommand{\Orth}{\textup{O}}
\newcommand{\Fun}{\mathop{\textup{RepFun}}}
\newcommand{\Irr}{\mathop{\textup{Irr}}}
\newcommand{\rk}{\mathop{\textup{rk}}{}}
\newcommand{\Lie}{\mathop{\textup{Lie}}}
\newcommand{\lpm}{\textup{lpm}}
\newcommand{\Inv}{\textup{Inv}}

\newcommand{\diag}{\textup{diag}}

\newcommand{\Span}{\mathop{\textup{Span}}}
\newcommand{\Image}{\mathop{\textup{Image}}}

\title{Finite matrix multiplication algorithms \\ from infinite groups}
\author{Jonah Blasiak\footnote{Department of Mathematics, Drexel University, Philadelphia, PA 19104, USA, \texttt{jblasiak@gmail.com}. Supported by NSF grant DMS-2154282.}, Henry Cohn\footnote{Microsoft Research New England, One Memorial Drive, Cambridge, MA 02142, USA, \texttt{cohn@microsoft.com}}, Joshua A. Grochow\footnote{Departments of Computer Science and Mathematics, University of Colorado Boulder, Boulder, CO 80309, USA, \texttt{jgrochow@colorado.edu}. Supported by NSF CAREER award CCF-2047756.},\\ Kevin Pratt\footnote{Department of Computer Science, Courant Institute of Mathematical Sciences, New York, NY 10012, USA, \texttt{kp2154@nyu.edu}. Supported by Subhash Khot's Simons Investigator Award.}, and
Chris Umans\footnote{Department of Computing and Mathematical Sciences, California Institute of Technology, Pasadena, CA 91125, USA, \texttt{umans@cs.caltech.edu}. Supported by a Simons Foundation Investigator Award.}
}
\date{}

\begin{document}

\maketitle

\begin{abstract}
The Cohn--Umans (FOCS '03) group-theoretic framework for matrix multiplication produces fast matrix multiplication algorithms from three subsets of a finite group $G$ satisfying a simple combinatorial condition (the Triple Product Property). 
The complexity of such an algorithm then depends on the representation theory of $G$. In this paper we extend the group-theoretic framework to the setting of \emph{infinite} groups. In particular, this allows us to obtain constructions in Lie groups, with favorable parameters, that are provably impossible in finite groups of Lie type (Blasiak, Cohn, Grochow, Pratt, and Umans, ITCS '23). Previously the Lie group setting was investigated purely as an analogue of the finite group case; a key contribution in this paper is a fully developed framework for obtaining bona fide matrix multiplication algorithms directly from Lie group constructions. 

As part of this framework, we introduce ``separating functions'' as a necessary new design component, and show that when the underlying group is $G=\GL_n$, these functions are polynomials with their degree being the key parameter. In particular, we show that a construction with ``half-dimensional'' subgroups and optimal degree would imply $\omega = 2$. 
We then build up machinery that reduces the problem of constructing optimal-degree separating polynomials to the problem of constructing a \emph{single} polynomial (and a corresponding set of group elements) in a ring of invariant polynomials determined by two out of the three subgroups that satisfy the Triple Product Property. This machinery combines border rank with the Lie algebras associated with the Lie subgroups in a critical way.

We give several constructions illustrating the main components of the new framework, culminating in a construction in a special unitary group that achieves separating polynomials of optimal degree, meeting one of the key challenges. The subgroups in this construction have dimension \emph{approaching} half the ambient dimension, but (just barely) too slowly. 
We argue that features of the classical Lie groups make it unlikely that constructions in these particular groups could produce nontrivial bounds on $\omega$ unless they prove $\omega = 2$. 
One way to get $\omega=2$ via our new framework would be to lift our existing construction from the special unitary group to $\GL_n$, and 
improve the dimension of the subgroups from $\frac{\dim G}{2} -\Theta(n)$ to $\frac{\dim G}{2} -o(n)$.
\end{abstract}

\section{Introduction}
Matrix multiplication is a fundamental algebraic operation with myriad algorithmic applications, and as such, determining its complexity is a central question in computational complexity. Since Strassen's 1969 discovery \cite{strassen} that one could beat the straightforward $O(n^3)$ method with one that used only $O(n^{\log_2 7}) = O(n^{2.81\dots})$ arithmetic operations, there has been a long line of work improving upper bounds on the complexity of matrix multiplication. It is standard to define the exponent $\omega$ of matrix multiplication as the smallest number such that $n \times n$ matrices can be multiplied using $n^{\omega + o(1)}$ arithmetic operations, and a somewhat surprising folklore conjecture is that $\omega=2$, which would mean matrices can be multiplied asymptotically almost as quickly as they can be added. The current best bound is that $\omega \leq 2.371339$ \cite{williams2024new}, and proving or disproving that $\omega=2$ remains a longstanding open question.

In \cite{cu2003}, an approach towards this problem was proposed based on embedding matrix multiplication into the multiplication operation in the group algebra of a finite group. Group algebra multiplication then reduces to the multiplication of block-diagonal matrices, where the sizes of these blocks are determined by the (typically well-understood) representation theory of the group. Ultimately, one hopes to reduce a single matrix multiplication to the multiplication of (many) smaller matrices. Within this very general framework, certain families of groups have subsequently been shown to have structural properties that prevent a reduction that would yield $\omega = 2$ using this framework \cite{BCCGNSUCapSetMM, BCCGUgroupsMM, blasiak2023matrix}. Other groups remain potentially viable, and the overall approach remains one of the two main lines of research towards improving upper bounds on $\omega$, the other being the traditional ``direct'' tensor methods (e.\,g., \cite{schonhage1981partial, strassen1988asymptotic, coppersmith1990matrix, williams2012multiplying,davie2013improved,le2014powers,alman2021refined,duan2023faster}), which also seem to be running up against several barrier results \cite{AmbainisFilmusLeGall,christandl2019barriersJournal,alman2023limits,CLGLZ}.

In more detail, the reduction to group algebra multiplication is possible when one identifies a finite group $G$ and sets $X,Y,Z \subseteq G$ satisfying the \emph{triple product property} (TPP): for any $x,x' \in X$, $y, y' \in Y$,  and $z,z' \in Z$, \[xx'^{-1}yy'^{-1}zz'^{-1} = 1_G \Longleftrightarrow x = x', y=y', z=z'.\] 
If $X, Y, Z \subseteq G$ satisfy the
TPP, then we can multiply two complex matrices $A$ and $B$, of sizes $|X| \times |Y|$ and $|Y| \times |Z|$, resp., as follows. Index $A$ with $X \times Y$, with $A[x,y]$ denoting the $x,y$ entry of $A$, and index $B$ with $Y \times Z$. Then we define the elements 
\[\overline{A} = \sum_{x,y} A[x,y] (xy^{-1}) \quad\textup{and}\quad \overline{B} =
  \sum_{y,z} B[y,z](yz^{-1})\]
of the group ring $\C[G]$ and observe that the TPP implies that
\begin{equation} \label{eq:embed}
\overline{A}\cdot \overline{B} = \sum_{x,z} (AB)[x,z](xz^{-1}) + E,
\end{equation}
where $E \in \C[G]$ is supported on $XY^{-1}YZ^{-1} \setminus XZ^{-1}$. It is a standard fact of representation theory that, as algebras, $\C[G] \cong \bigoplus_i M_{d_i}(\C)$, where $M_{d_i}$ denotes the ring of $d_i \times d_i$ matrices, the sum is over the irreducible representations of $G$, and the $d_i$ are the dimensions of those representations. This leads to the inequality:

\begin{theorem}[\protect{\cite[Theorem 4.1]{cu2003}}]
\label{thm:main-TPP-finite-groups}
If $X, Y, Z$ satisfy the TPP in a finite group $G$, then
\[(|X|\,|Y|\,|Z|)^{\omega /3} \le \sum_i d_i^\omega,\]
where $d_i$ are the dimensions of the irreducible representations of $G$. 
\end{theorem}

The upper bound on $\omega$ from Theorem~\ref{thm:main-TPP-finite-groups} depends on a trade-off between the size of the matrix multiplication that can be embedded into $\C[G]$ (reflected in $|X|, |Y|, |Z|$) and the dimensions of the irreducible representations of $G$. In abelian groups, the latter are optimal: all $d_i$ are 1. However, already in \cite{cu2003} it was observed that abelian groups cannot do better than the trivial construction $X=G$ and $Y=Z=\{1\}$, and thus cannot yield any bound better than the trivial bound $\omega \leq 3$. It was shown in \cite{cohn2005group} that non-abelian groups can achieve highly nontrivial bounds on $\omega$ (including many of the state-of-the-art bounds over the last decade), but obtaining such bounds requires a careful interplay between the size of the construction and the representation dimensions. Several families of groups have been shown not to admit constructions for which Theorem~\ref{thm:main-TPP-finite-groups} would give $\omega=2$, although many families of groups remain possibilities. 

Because of the difficulty of finding such constructions, it is useful to look for other potential sources of examples of groups and group-like objects that could yield such constructions. Cohn and Umans \cite{cu2003} gave a construction of a TPP triple in the infinite group $\GL_n(\R)$, despite having, at the time, no way of getting a (finite) matrix multiplication algorithm from such a construction.

One natural approach (that ends up not working) is to take a construction in a Lie group and try to transfer it to a finite group of Lie type, such as taking a construction in $\GL_n(\R)$ and attempting an analogous construction in $\GL_n(\F_q)$. Indeed, a construction in $\GL_n(\R)$ was given in \cite{cu2003} using the lower unitriangular, orthogonal, and upper unitriangular subgroups; this inspired a nontrivial TPP in the finite group $\SL_2(\F_q)$, where the orthogonal group is replaced by matrices of the form $\begin{pmatrix} 1+a & a \\ -a & 1-a \end{pmatrix}$. However, in that example, we have $|X|=|Y|=|Z|=q$, but $\SL_2(\F_q)$ has irreducible representations of dimension $q+1$, and reducing a $q \times q$ matrix multiplication to a $(q+1) \times (q+1)$ matrix multiplication doesn't give any bound on $\omega$. More generally, in \cite{blasiak2023matrix} it was shown that one cannot achieve $\omega=2$ via Theorem~\ref{thm:main-TPP-finite-groups} in any finite groups of Lie type, ruling out any such construction in $\GL_n(\F_q)$ or similar groups (though possibilities remain open to use related groups such as direct products of finite groups of Lie type). Thus, the Lie-type constructions remained only an analogy.

\subsection{First main contribution: algorithms from infinite groups}
In this paper, one of our main innovations is to extend the group-theoretic framework \cite{cu2003} to allow finite matrix multiplication algorithms from constructions in arbitrary---even infinite---groups. To achieve this, rather than using the entire group algebra $\C[G]$, which is infinite-dimensional when $G$ is infinite, we focus only on sets of functions $G \to \C$ that (1) are linearly computable from a finite set of finite-dimensional representations of $G$ (even when $G$ is infinite) and (2) ``separate'' the elements in the group algebra that are in the linear span of $XY^{-1} Y Z^{-1}$ (the support of \eqref{eq:embed}), in a sense made precise below (Definition~\ref{def:sep}). Our first main theorem in this generalized framework is then:

\begin{maintheorem}[{Theorem~\ref{thm:framework}}] \label{thm:main}
Let $G$ be a group (not necessarily finite), with finite subsets $X,Y,Z$ satisfying the TPP. If $R_{\textup{\textup{sep}}}$ is a set of finite-dimensional complex representations of $G$ whose matrix entries  separate $XY^{-1}YZ^{-1}$ (see Definition~\ref{def:sep}), then 
\[
\left(|X|\, |Y|\, |Z| \right)^{\omega/3} \leq \sum_{\rho \in R_{\textup{\textup{sep}}}} (\dim \rho)^\omega.
\]
\end{maintheorem}

It is even conceivable that this result could be used to improve the bounds from known constructions of TPPs in finite groups, by using only a subset of the group's irreducible representations rather than all of them.

But perhaps the main payoff of Theorem~\ref{thm:main} is that it allows, for the first time, the derivation of matrix multiplication algorithms from infinite groups. This opens a huge variety of potential constructions to explore, even beyond those in Lie groups that will be the focus of the rest of the paper.

\subsection{Second main contribution: quantitative targets for proving $\omega = 2$ in classical Lie groups}
Another main contribution in this paper is to develop a series of tools and techniques, and to identify key targets to aim for, for getting good constructions using Theorem~\ref{thm:main} in Lie groups such as $\GL_n(\R)$, $\GL_n(\C)$, or the unitary group $\U_n$. Lie groups are defined as groups that are also smooth manifolds, and where the group multiplication and inverse are continuous in terms of the manifold topology. In fact, all of our constructions in this paper will be in one of these three groups, though much of the machinery we develop works for general matrix Lie groups, and it would be interesting to explore constructions in other Lie groups such as symplectic groups, orthogonal groups, exceptional simple Lie groups, or nilpotent or solvable Lie groups.

Before coming to the constructions, we highlight how the framework of Theorem~\ref{thm:main}  allows us to take the analogy from \cite{cu2003, blasiak2023matrix}, and turn it into a formal implication whose conclusion is a bound on $\omega$. To briefly recall the analogy: elementary arguments show that any TPP triple in a finite group must satisfy $|X|\,|Y|\,|Z| \leq |G|^{3/2}$, called the \emph{packing bound} because of the nature of the proof. Because $\sum d_i^2 = |G|$ in finite groups, it follows that any sequence of constructions that achieves $\omega=2$ via Theorem~\ref{thm:main-TPP-finite-groups} (not our new Theorem~\ref{thm:main}) must asymptotically \emph{meet the packing bound}, in the sense that $|X|\,|Y|\,|Z| \geq |G|^{3/2 - o(1)}$. The analogy studied in \cite{cu2003, blasiak2023matrix} is to think of Lie subgroups of dimension $d$ as ``roughly corresponding'' to finite subsets of size $q^d$. Under this analogy, if $X_n,Y_n,Z_n$ are families of Lie subgroups of Lie groups $G_n$ that satisfy the TPP, then meeting the packing bound is analogous to the condition
\begin{equation} \label{eq:Lie-packing-bound}
\dim X_n + \dim Y_n + \dim Z_n \geq (3/2 - o_n(1)) \dim G_n.
\end{equation}
A simple construction in the original Cohn--Umans paper \cite[Theorem 6.1]{cu2003} shows this is indeed possible (with additional such constructions developed in \cite{blasiak2023matrix}), and this construction forms the basis for a running example we will use throughout this paper to illustrate the development of our new framework.

\bigskip
\hrule
\vspace{-.3in}
\begin{adjustwidth}{.2in}{.2in}
\subsection*{\center Running example in $\GL_n(\R)$}
\begin{theorem}[{\cite[Theorem~6.1]{cu2003}}]
Let $ G = \GL_n(\R)$ and let $X$ be the subgroup of lower unitriangular matrices, $Y = \Orth_n(\R)$ (the orthogonal matrices), and $Z$ be the subgroup of upper unitriangular matrices. Then the triple $X, Y, Z$ satisfies the TPP.
\label{thm:running-ex-1}
\end{theorem}

The dimension of $G$ is $n^2$, and the dimension of each of the three subgroups is $n^2/2- n/2$, so this construction meets the packing bound in the sense of \eqref{eq:Lie-packing-bound} \cite{blasiak2023matrix}.

\begin{proof}
For $L\in X$, $R \in Y$, and $U \in Z$ we must show $LRU = I$ implies $L = R = U = I$. The proof is by induction on $n$. The statement follows trivially when $n = 1$. For $n > 1$, consider the top-left entry of $LRU$ and observe that it equals the top-left entry of $R$. Since $LRU = I$ the top-left entry is $1$ and hence the top-left entry of $R$ is 1, and then the remainder of the first column and top row of $R$ must be zeros (since rows and columns of $R$ have length 1). At this point the left column of $L$ equals the left column of $LRU = I$ and the top row of $R$ equals the top row of $LRU = I$. If we define $L', R', U'$ to be the lower-right $(n-1) \times (n-1)$ submatrices of $L, R$, and $U$, respectively, we find that these are lower unitriangular, orthogonal, and upper unitriangular matrices for which $L'R'U' = I$, which implies $L' = R' = U' = I$ by induction. Combined with our observation about the left column and top row of $L, R$ and $L$, we conclude that $L =R =U = I$ as required.  
\end{proof}
\end{adjustwidth}
\hrule
\bigskip

But there still remains the issue of how to use such a construction, even in concert with Theorem~\ref{thm:main}, to get upper bounds on $\omega$. And for this we require a deeper dive into representation theory, which will lead us to the key quantitative goals for these constructions.

In the case of the groups $\GL_n(\R)$, $\GL_n(\C)$,  and $\U_n$, rather than focusing on choosing arbitrary collections of representations to play the role of $R_{\textup{sep}}$ in Theorem~\ref{thm:main}, we can exploit a relationship between the representations of these groups and the set of all degree-$d$ polynomials. Namely, the irreducible representations of these groups are indexed by integer partitions into at most $n$ parts, and the matrix entries of the representations corresponding to partitions of $d$, taken all together, span precisely the set of all degree-$d$ polynomial functions on the group. Careful quantitative estimates then lead us to key targets for the degree of functions that separate out the elements of $XY^{-1} YZ^{-1}$ compared to the size of the finite TPP construction:

\begin{maintheorem}[{Corollary~\ref{cor:irrepbound}, summarized}] \label{thm:main-irrep-bound}
If $X_{q,n},Y_{q,n},Z_{q,n} \subseteq \GL_n(\C)$ (or $\U_n$) satisfy the TPP and have sizes at least $q^{n^2/2 - o_n(n)}$, and there are separating polynomials for $(X_{q,n}, Y_{q,n}, Z_{q,n})$ of degree at most $q^{1 + o_q(1)}$, then Theorem~\ref{thm:main} implies $\omega=2$.
\end{maintheorem}

Note how our new framework (Theorem~\ref{thm:main}) has allowed us to take the analogy above whereby $d$-dimensional subgroups correspond to finite sets of size $q^d$, and turn it into a theorem that actually implies a bound on $\omega$ out of any such construction, rather than merely being an analogy. However, in this setup, as $\dim G = n^2$, we see that the bound we need on the size of the TPP triple is not $|X|\,|Y|\,|Z|\geq q^{(3/2 - o_n(1))\dim G}$ as suggested in the previous Lie analogy of the packing bound \cite{cu2003,blasiak2023matrix}, but is slightly tighter, of the form $q^{(3/2-o_n(1/n))\dim G}$, and the example in Theorem~\ref{thm:running-ex-1} above falls short of this latter bound.

\textbf{Our challenge} is thus to find TPP constructions in $\GL_n$ or $\U_n$ that meet the bounds of Theorem~\ref{thm:main-irrep-bound}: three subsets in $\GL_n$ or $\U_n$ satisfying the TPP, of size at least $q^{n^2/2 - o_n(n)}$, and admitting separating polynomials of degree at most $q^{1 + o_q(1)}$. 

\begin{remark} \label{rmk:dim}
Although our framework is more general, in all our constructions in this paper, we will construct the finite sets $X_{q,n}, Y_{q,n}, Z_{q,n}$ as subsets of  submanifolds $X_n, Y_n, Z_n$ that satisfy the TPP, where $|X_{q,n}| \sim q^{\dim X_n}$, $|Y_{q,n}| \sim q^{\dim Y_n}$, and $|Z_{q,n}| \sim q^{\dim Z_n}$. In that setting we can make two further remarks. The first is that the TPP implies that the pairwise maps $X_n \times Y_n \to G_n$ given by $(x,y) \mapsto xy$ are injective, and by smoothness of the multiplication map, this implies that $\dim X_n + \dim Y_n \leq \dim G_n$. In particular,  the best possible is $\dim X_n = (1/2) \dim G_n$. (One of the three sets could have larger dimension, but then the arithmetic mean of the three dimensions would be less than $(1/2) \dim G_n$ and the bound on $\omega$ would be worse.) We will see that sets of this dimension would be just good enough to yield $\omega=2$ if the representation dimensions were favorable. 

The second remark is that, over an algebraically closed field, if $G_n$ is an algebraic group and $X_n, Y_n, Z_n$ are not just submanifolds but also subvarieties, then $\dim X_n + \dim Y_n + \dim Z_n \leq \dim G_n$ \cite[Theorem~4.7]{blasiak2023matrix}. In particular, when $|X_{q,n}| \sim q^{\dim X_n}$ and similarly for $Y$ and $Z$, over an algebraically closed field no nontrivial bound on $\omega$ (strictly less than $3$) can be gotten from three sub\emph{varieties}. However, subvarieties over non-algebraically closed fields are not in general subject to this restriction; indeed, our constructions in this paper will mostly be \emph{real} algebraic subvarieties.
\end{remark}

\subsection{Third main contribution: optimal degree using border rank, Lie algebras, and invariant theory}

Our third main contribution is to show how to very nearly meet the conditions of Theorem~\ref{thm:main-irrep-bound}, using border rank, Lie algebras, and invariant theory. We also believe these techniques will have further uses. Our main theorem coming out of these constructions is:

\begin{maintheorem}[{Summary of Theorem~\ref{thm:construction}}] \label{thm:main-construction}
For any $n$ and $q$, there are three subsets $X_q, Y_q, Z_q \subseteq \U_n$, all of size at least $q^{n^2/4-n/4}$, which satisfy the TPP and admit border-separating polynomials of degree $O(q)$.
\end{maintheorem}

Note that this falls short of the conditions needed for Theorem~\ref{thm:main-irrep-bound} in only two ways: we get a construction of size $q^{(1/2)\dim G - \Theta(n)}$ where $G = \U_n$, whereas Theorem~\ref{thm:main-irrep-bound} would require both that the construction be in $\GL_n$, and that it approach half the dimension just slightly faster, with sets of size $q^{(1/2)\dim \GL_n - o_n(n)}$, rather than our current $q^{(1/2) \dim \U_n - \Theta(n)}$. In addition to the sizes being very nearly right, the degree bound \emph{does} satisfy the degree bound required by Theorem~\ref{thm:main-irrep-bound} (even slightly better than is needed: we get $O(q)$ whereas Theorem~\ref{thm:main-irrep-bound} only needs degree at most $q^{1 + o_q(q)}$). We additionally point out that in the unitary groups, except for lower-order terms in the exponent, both our size and degree are tight: the construction comes from submanifolds, which can have dimension at most $(1/2)\dim U_n$ (Remark~\ref{rmk:dim}), and unconditionally at such sizes one cannot have separating polynomials degree $q^{1-c}$ for any $c > 0$ (without implying the contradiction $\omega < 2$).

While in principle all one needs here is a sequence of finite sets $X_q, Y_q, Z_q$ for infinitely many $q$, an appealing way to get such a sequence is to find Lie subgroups or submanifolds $X,Y,Z \subseteq \GL_n(\C)$---as in Theorem~\ref{thm:running-ex-1}---and then let $X_q$ be some nicely constructed finite subset of $X$, etc. For example, when $X$ is the subgroup of lower unitriangular matrices, we can take $X_q$ to be the the lower unitriangular matrices with entries in $[0,q] \cap \Z$. And indeed, this is how the construction of Theorem~\ref{thm:main-construction} proceeds.

To get our constructions, we will combine three additional ingredients on top of Theorems~\ref{thm:main} and \ref{thm:main-irrep-bound}: Lie algebras, border rank, and invariant theory. Here we give a brief overview of these ingredients and how they mix together.

\textbf{Lie algebras.} In Sophus Lie's original development of Lie groups in the late 1800s, he realized that many questions about these continuous groups can be reduced to simpler questions of linear algebra, by focusing on the corresponding Lie algebras, which are, in particular, vector spaces, rather than more complicated manifolds. The Lie algebra $\Lie(G)$ associated to a Lie group $G$ is ``just'' the tangent space to $G$ (remember $G$ is also a manifold) at the identity element. The Lie algebra is then a vector space. While the group structure of $G$ induces an algebraic structure on $\Lie(G)$, in this paper we will have no need of its algebraic structure. We will only need a few basic facts (see, e.\,g., \cite{Pollatsek, baker, fulton1991representation}):
\begin{itemize}
\item The Lie algebra of $\GL_n(\C)$ is $M_n(\C)$, the space of all $n \times n$ complex matrices.

\item The Lie algebra of the orthogonal group $O_n(\R)$ consists precisely of all skew-symmetric real matrices.

\item The Lie algebra of the unitary group consists of the skew-Hermitian matrices.

\item If $G$ is a matrix Lie group---a Lie group that is a subgroup and submanifold of $\GL_n(\C)$---then its Lie algebra $\Lie(G)$ is a linear subspace of matrices. And if $A \in \Lie(G)$, then for all sufficiently small $\varepsilon > 0$, $\exp(\varepsilon A)$ (using the ordinary power series for the matrix exponential) is in $G$.
\end{itemize}
As with many problems on Lie groups, we would like to take advantage of the simpler, linear-algebraic nature of Lie algebras in our constructions.

\textbf{Border rank.} It turns out that the key tool for using Lie algebras in our setting is the concept of \emph{border rank}. Although known to  Terracini 100 years ago, border rank was rediscovered in the context of matrix multiplication by Bini \emph{et al.} \cite{bini,bini1979n2}. Bini had developed computer code to search for algebraic algorithms for matrix multiplication, and some of the coefficients in his numerical calculations kept going off to infinity. At first he thought this was an error in his code, but in fact it reflects the fundamental phenomenon of border rank: for each fixed size $n_0$, it is possible that there is a sequence of algorithms, none of which correctly multiply $n_0 \times n_0$ matrices, but which, \emph{in the limit}, in fact do so. If the algorithms in the sequence use only $r$ non-constant multiplications, we say that matrix multiplication has border rank at most $r$. The border rank is always at most the ordinary rank, and it turns out that the exponent of matrix multiplication is the same whether measured with ordinary rank or border rank. Border rank has played an important role in essentially all newly developed matrix multiplication algorithm since then.

A bit more formally, a ``border algorithm'' for matrix multiplication can be viewed as a single bilinear algorithm that has coefficients that are Laurent series in $\varepsilon$---that is, power series that allow finitely many negative powers of $\varepsilon$ as well---and such that it computes matrix multiplication in the limit as $\varepsilon \to 0$, that is, it computes a function of the form $(A,B) \mapsto AB + O(\varepsilon)$. (Note that, despite the algorithm itself being allowed to contain $1/\varepsilon$ in its intermediate operations, corresponding to Bini's coefficients that were going off to infinity, the function computed at the end should have such negative powers of $\varepsilon$ cancel.) 

\textbf{Combining Lie algebras and border rank.} It turns out that combining border rank with Lie algebras is very natural; here we exhibit just two advantages to doing so. If $X,Y,Z \subseteq \GL_n(\C)$ are Lie subgroups that satisfy the TPP (as in Theorem~\ref{thm:running-ex-1}), then we can take advantage of the simple linear-algebraic nature of their Lie algebras to help construct finite subsets of $X,Y,Z$. An example we will return to is that if $Y = O_n(\R)$ is the orthogonal group, it is a bit tricky to choose $q^{\dim O_n}$ many elements of $Y$ in a principled way directly. But since $\Lie(Y)$ consists of all the skew-symmetric matrices, we can get a finite subset of $Y$ by simply considering 
\[
Y'_\varepsilon = \{\exp(\varepsilon A) : A \text{ skew-symmetric with all } A_{ij} \in [0,q] \cap \Z\}
\]
for $\varepsilon > 0$ sufficiently small.

A second advantage can be gotten by not choosing $\varepsilon$ as above to be a fixed but small value, but rather allowing to the $\varepsilon$ used in the expression $\exp(\varepsilon A)$ to be the same as the parameter $\varepsilon$ used in the definition of border rank. In this setting, instead of finding a set of functions that exactly separates $XY^{-1} Y Z^{-1}$ according to Definition~\ref{def:sep}, we can find functions that do so only up to $O(\varepsilon)$ (Definition~\ref{def:sep-border}). This both gives us more freedom in the construction, and combines very naturally with the Lie algebraic construction suggested above. Namely, when all the elements of  $X',Y',Z'$ are of the form $\exp(\varepsilon A)$ for various $A$ in their Lie algebras, we see that an expression of the form $x y^{-1} y' z^{-1}$ becomes 
\[
\exp(\varepsilon A) \exp(-\varepsilon B) \exp(\varepsilon B') \exp(-\varepsilon C) = I + \varepsilon(A - B + B' - C) + O(\varepsilon^2).
\]
Our border-separating polynomials can then directly access $A-B+B'-C$ by subtracting off $I$ and dividing by $\varepsilon$; this leaves additional $O(\varepsilon)$ terms, but in the border setting that is still allowed. Thus combining Lie algebras and border rank lets us shift the problem from finding separating polynomials in the entries of a product of matrices and their inverses to finding (border-)separating polynomials in the entries of the simpler linear combination $A-B+B'-C$.

However, the linear combination $A-B+B'-C$ still ``mixes'' entries from the three subalgebras, and this causes some difficulty in the the task of designing (border-)separating polynomials. Our final ingredient is to use invariant theory to simplify this task even further.

\subsection{Fourth main contribution: leveraging invariant polynomials}
If we restrict our attention to polynomials $p(M)$ that are invariant under left multiplication by $X$ and right multiplication by $Z^{-1}$, that is, $p(xMz^{-1}) = p(M)$ for all $x \in X, z \in Z$, we can get direct access to the $B'-B$ term above. Namely, if $p$ is such an invariant polynomial, then
\[
p(\exp(\varepsilon A) \exp(-\varepsilon B) \exp(\varepsilon B') \exp(-\varepsilon C)) = p(\exp(-\varepsilon B) \exp(\varepsilon B')) = p(I + \varepsilon(-B+B') + O(\varepsilon^2)),
\]
where the first equality occurs because $\exp(\varepsilon A)$ is in $X$ and $\exp(\varepsilon C)$ is in $Z$, and $p$ is invariant under the action of $X$ and $Z$. We codify this idea into the following key lemma, which is used in the proof of Theorem~\ref{thm:main-construction}. We state it in a simplified form, which is not correct as stated but captures the spirit of a special case. See Lemma~\ref{lem:split} for the full and correct statement.

\begin{mainlemma}[{Oversimplified version of Lemma~\ref{lem:split}}] \label{lem:main}
Suppose $X,Y,Z$ are three Lie subgroups of $\GL_n(\C)$ that satisfy the TPP. If $Y'_\varepsilon$ is a finite set of $\varepsilon$-parametrized families in $Y$, and there is a polynomial $p_\varepsilon(g)$ that is invariant under left multiplication by $X$ and right multiplication by $Z$, such that
\[
p_\varepsilon(g) = \begin{cases}
1 + O(\varepsilon) & \text{ if } g = I \\
0 + O(\varepsilon) & \text{ if } g \in (Y'_\varepsilon)^{-1} Y'_{\varepsilon} \setminus \{I\}
\end{cases}
\]
then there are finite subsets $X'_\varepsilon, Y'_\varepsilon$ of $\varepsilon$-parametrized elements of $X$ and $Z$, of sizes $q^{\dim X}$ (resp., $q^{\dim Z}$) and border-separating polynomials for $(X'_\varepsilon, Y'_\varepsilon, Z'_\varepsilon)$ of degree $\deg p + O(q)$.
\end{mainlemma}

Thus, once we have a TPP construction of the ``right'' dimension (e.\,g., according to Theorem~\ref{thm:main-irrep-bound}), Lemma~\ref{lem:main} implies that instead of finding appropriate finite subsets of $X,Y,Z$ and a \emph{set} of separating polynomials, to get $\omega=2$ all we have to do is find an appropriate finite subset of $Y$ and a \emph{single} polynomial $p$ as in the lemma, whose degree is $O(q)$.

When $X$ and $Z$ are well-studied subgroups, the ring of their invariant polynomials is often well understood, and this represents a significant simplification of the construction problem. The full detail of Lemma~\ref{lem:split} is more complicated, but the additional complications in the statement of the lemma give us even further simplifications of the construction problem, which we take advantage of in our constructions in the running example and in Section~\ref{sec:construction}. We also expect the techniques of this lemma to have further uses for additional constructions in the future.

\subsection{Paper outline}
In Section~\ref{sec:framework-thm} we give the necessary definitions and proof of Theorem~\ref{thm:main}, and in Section~\ref{sec:border} we give its border rank version. In Section~\ref{sec:irrepbound} we discuss the needed background on the representation theory of $\GL_n(\C)$---much of which we can use in a black-box fashion---and prove Theorem~\ref{thm:main-irrep-bound}. In Section~\ref{sec:invariant} we show how to combine invariant theory, border rank, and Lie algebras to simplify the construction task; the key here is Lemma~\ref{lem:split}. In Section~\ref{sec:construction} we use the preceding ideas to give the construction that proves Theorem~\ref{thm:main-construction}. Finally, in Section~\ref{sec:conclusion} we conclude with our outlook and several open problems suggested by our framework.

Throughout the development of the framework in Section~\ref{sec:framework}, we continue the example from Theorem~\ref{thm:running-ex-1} as a running example, and we show how each piece of the framework can be realized in that case.

The paper contains complete proofs of all claims, though a small number are deferred to Appendix~\ref{appendix} to help maintain the flow.

\section{A group-theoretic framework for infinite groups}
\label{sec:framework}
In this section we describe the group-theoretic framework for obtaining matrix multiplication algorithms in infinite groups. The basic framework is in the next subsection, followed by a border-rank version for the important case of Lie groups. Subsection~\ref{sec:irrepbound} proves some key bounds on the dimensions of irreducible representations of $\GL_n$ and $\U_n$, which are the containing groups for our constructions in this paper. Finally subsection~\ref{sec:invariant} describes machinery that reduces the main design task to finding a single separating polynomial in an invariant ring. These components come together in Section~\ref{sec:construction}, where we give a construction achieving optimal degree separating polynomials. 

\subsection{Algorithms from the TPP in infinite groups}
\label{sec:framework-thm}

Let $G$ be a group---not necessarily finite---with finite subsets $X, Y,Z$ satisfying the TPP. Then we can embed matrix multiplication into the group algebra $\C[G]$ as in the case of finite groups. Note that $\C[G]$, where $G$ may be infinite, consists of formal sums $\sum_{g \in G} \alpha_g g$, but now where each such sum has at most finitely many nonzero terms. Multiplication is as in the finite case.

To multiply a complex matrix $A$ of size $|X| \times |Y|$ by a complex matrix $B$ of size $|Y| \times |Z|$, we follow the approach described in the introduction. Indexing $A$ by $X \times Y$ and $B$ by $Y \times Z$, we define elements
\[
\overline{A} = \sum_{x \in X,y \in Y} A[x,y] (xy^{-1}) \quad\textup{and}\quad \overline{B} =
  \sum_{y \in Y,z\in Z} B[y,z](yz^{-1})
\]
of $\C[G]$ and observe that, as in the finite group case, the TPP implies that
\begin{equation} \label{eqn:group-algebra}
\overline{A}\cdot \overline{B} = \sum_{x \in X,z\in Z} (AB)[x,z](xz^{-1}) + E,
\end{equation}
where $E \in \C[G]$ is supported on $XY^{-1}YZ^{-1} \setminus XZ^{-1}$.

However, as the group algebra is no longer finite-dimensional when $G$ is infinite, it is not immediately clear that the multiplication in the group algebra expressed in \eqref{eqn:group-algebra} can be carried out by a finite algorithm, even despite the fact that all the sums involved are themselves finite. One of our key new innovations is to introduce a new viewpoint on such a construction that will enable us to get a finite bilinear algorithm out of the above construction.

\begin{definition} \label{def:sep}
Given subsets $X,Y,Z \subseteq G$ of a group $G$, a \emph{set of separating functions for $(X,Y,Z)$} is a collection of functions $\{f_{x,z}\colon G \to \C \mid x \in X, z \in Z\}$ such that
\[
f_{x,z}(g) = \begin{cases} 1 & \text{if } g=xz^{-1}, \text{ and}\\
                 0 & \text{if } g \in XY^{-1} Y Z^{-1}
                                       \setminus\{xz^{-1}\}.\end{cases}
\]
\end{definition}

Such functions always exist, but for them to be useful in matrix multiplication algorithms we need them to be in some sense ``simple.'' To formulate the relevant notion---in which ``simple'' will ultimately imply ``computable by smaller matrix multiplications''---we recall a standard definition from representation theory. Given a finite-dimensional  representation $\rho \colon G \to \GL_n(\C)$, a \emph{representative function} of $G$ associated to $\rho$ (see, e.\,g., Procesi's book \cite[\S 8.2]{procesi}\footnote{The notion of representative function we use here is fairly standard; if one consults Procesi's book, one will see that his book works in the setting where $G$ is a topological group, and requires the representations involved to be continuous, but the definition we use here works just as well. When $G$ is a Lie group, as in our constructions later in the paper, the representations we use will in fact be continuous in the natural manifold topology on $G$ rather than the discrete topology.}) is any function $G \to \C$ that is in the linear span of the functions $\{ (g \mapsto \rho(g)_{i,j}) \mid i,j \in [n]\}$. Note that if $\rho,\rho'$ are equivalent representations, they have the same set of representative functions. If $\mathcal{I}$ is a set of representations, we write $\Fun(\mathcal{I})$ for the $\C$-linear span of all the representative functions of the representations in $\mathcal{I}$.

\begin{theorem} \label{thm:framework}
Let $G$ be a group (not necessarily finite), with finite subsets $X,Y,Z$ satisfying the TPP. If $R_{\textup{sep}}$ is a finite set of finite-dimensional complex representations of $G$ such that $\Fun(R_{\textup{sep}})$ contains a set of separating functions for $(X,Y,Z)$, then 
\[
\left(|X| \, |Y| \, |Z| \right)^{\omega/3} \leq \sum_{\rho \in R_{\textup{sep}}} (\dim \rho)^\omega.
\]
In particular, for $D = \sum_{\rho \in R_{\textup{sep}}} (\dim \rho)^2$ and $d_{max} = \max\{\dim \rho : \rho \in R_{\textup{sep}}\}$,
\[
\left(|X| \, |Y| \, |Z| \right)^{\omega/3}\leq D \cdot d_{max}^{\omega-2}.
\]
\end{theorem}

If $G$ is a finite group, we may take $R_{\textup{sep}}=\Irr(G)$, the set of all irreducible representations of $G$, as $\Fun(\Irr(G))$ is the collection of \emph{all} functions $G \to \C$ when $G$ is finite. In that case, we also have $D = |G|$, recovering the original theorem for finite groups \cite[Theorem 4.1]{cu2003} as a special case of Theorem~\ref{thm:framework}.

\begin{remark}
In families of groups $\{G_i\}$ with ``rapidly growing'' irreducible representation dimensions, this theorem exhibits an ``all-or-nothing'' phenomenon, which we explain here. For the constructions in this paper and other natural constructions in classical Lie groups, $d_{max} \ge D^{1/2 - g(i)}$ for some $g \in o(1)$ (and this is what we mean by ``rapidly growing irreducible representation dimensions'').
If we let $V = (|X||Y||Z|)^{1/3}$, it is clear that we must have $V >d_{max}$ for the inequality to imply {\em any} upper bound on $\omega$. Together with the observation that $d_{max}$ approaches $D^{1/2}$, this means that to prove any upper bound on $\omega$ we must have $V \ge D^{1/2 - f(i)}$ for some $f \in o(1)$. The inequality from the theorem then becomes
\[D^{(1/2 - f(i))\omega} \le V^\omega \le D\cdot d_{max}^{\omega - 2} \le D\cdot D^{(1/2 - g(i))(\omega - 2)}.\]
Taking the base-$D$ logarithm of both sides, we get \[\omega/2 - f(i)\omega \le \omega/2 - g(i)\omega + 2g(i),\] 
which implies $\omega \le 2\left (\frac{g(i)}{g(i)-f(i)} \right )$. In constructions the natural thing to do is to ensure $f(i)$ goes to zero more rapidly than $g(i)$, in which case $\omega = 2$. While it is {\em possible} that $f(i)$ could approach $cg(i)$ for some $c > 0$ and yield an upper bound  on $\omega$ strictly between $2$ and $3$, this would require detailed knowledge of the lower order terms of $d_{max}$ and/or very fine control of the lower order terms of $V$, which we do not typically have.
\end{remark}

\begin{proof}[Proof of Theorem~\ref{thm:framework}]
Let $X,Y,Z,R_{\textup{sep}}$ be as in the statement, and let $\{f_{x,z} \mid x \in X, z \in Z\}$ be the claimed set of separating functions contained in $\Fun(R_{\textup{sep}})$. For $f \colon G \to \C$, let $\overline{f} \colon \C[G] \to \C$ denote its unique linear extension to the group ring $\overline{f}(\sum \alpha_g g) := \sum \alpha_g f(g)$; as these sums have only finitely many nonzero terms by definition of the group ring (even when $G$ is infinite), there is no issue of convergence. Applying $\overline{f_{x,z}}$ to \eqref{eqn:group-algebra} gives
\begin{equation} \label{eqn:group-algebra-fxz}
\overline{f_{x,z}}(\overline{A} \cdot \overline{B}) = \sum_{x' \in X, z' \in Z} (AB)[x',z'] f_{x,z}(x'(z')^{-1}) + \overline{f_{x,z}}(E)
= (AB)[x,z] + 0,
\end{equation}
for $f_{x,z}$ is zero on all group elements of $X (Y)^{-1} Y (Z)^{-1}$ other than $xz^{-1}$, and \eqref{eqn:group-algebra} is entirely supported on $X (Y)^{-1} Y (Z)^{-1}$. To turn this into a finite algorithm, we will show that we can essentially do exactly the application of $\overline{f_{x,z}}$ in \eqref{eqn:group-algebra-fxz}, but working only in the representations in $R_{\textup{sep}}$ rather than working in the full group ring.

For a representation $\rho \in R_{\textup{sep}}$, let $\overline{\rho}$ denote the unique linear extension of $\rho$ to $\C[G]$: $\overline{\rho}(\sum \alpha_g g) = \sum \alpha_g \rho(g)$, and let $\rho_{i,j}(g)$ be the $(i,j)$ entry of the matrix $\rho(g)$, which we think of as a function $\rho_{i,j} \colon G \to \C$. Since $f_{x,z}$ is in $\Fun(R_{\textup{sep}})$ by assumption, we can write $f_{x,z}$ as a $\C$-linear combination of the functions $\{\rho_{i,j} \mid \rho \in R_{\textup{sep}}, \ i,j \in [\dim \rho]\}$, say $f_{x,z} = \sum_{\rho \in R_{\textup{sep}}} \sum_{i,j \in [\dim \rho]} M_{x,z,i,j} \rho_{i,j}$. Then we define $\widehat{f_{x,z}}(\rho)$ to be the matrix $M_{x,z,*,*}$, i.e.,
\[
f_{x,z}(g) = \sum_{\rho \in R_{\textup{sep}}} \sum_{i,j \in [\dim \rho]} \widehat{f_{x,z}}(\rho)_{i,j} \rho_{i,j}(g)
\]
for all $x \in X, z \in Z, g \in G$. Finally, extending linearly and applying $\overline{f_{x,z}}$ to $\overline{A} \cdot \overline{B}$ as in \eqref{eqn:group-algebra-fxz}, we get
\begin{align*}
(AB)[x,z] & = \overline{f_{x,z}}(\overline{A} \cdot \overline{B}) \\
& = \sum_{\rho \in R_{\textup{sep}}} \sum_{i,j \in [\dim \rho]} \widehat{f_{x,z}}(\rho)_{i,j} \overline{\rho_{i,j}}(\overline{A} \cdot \overline{B}) \\
& = \sum_{\rho \in R_{\textup{sep}}} \langle \widehat{f_{x,z}}(\rho), \overline{\rho}(\overline{A} \cdot \overline{B}) \rangle \\
& = \sum_{\rho \in R_{\textup{sep}}} \langle \widehat{f_{x,z}}(\rho), \overline{\rho}(\overline{A}) \cdot \overline{\rho}(\overline{B}) \rangle
\end{align*}
The summation and inner product are linear functions whose coefficients are independent of the input matrices $A,B$, so they are ``free'' in a bilinear algorithm. 

The product $\overline{\rho}(\overline{A}) \cdot \overline{\rho}(\overline{B})$ is a product of $d_\rho \times d_\rho$ matrices, where $d_\rho = \dim \rho$. Hence, the bilinear (i.e., tensor) rank of the preceding expression gives the following bound. Using $\rk$ to denote the tensor rank, and $\langle n,m,p \rangle$ to denote the tensor corresponding to matrix multiplication of an $n \times m$ matrix times and $m \times p$ matrix, we get
\[
\rk \langle |X|, |Y|, |Z| \rangle \leq \sum_{\rho \in R_{\textup{sep}}} \rk \langle d_\rho, d_\rho, d_\rho \rangle.
\]
Exactly as in the finite group case \cite{cu2003}, by symmetrizing we effectively get a square matrix multiplication of size $(|X|\,|Y|\,|Z|)^{1/3}$ on the left side, and by the tensor power trick the right side here can be replaced by $\sum_{\rho \in R_{\textup{sep}}} (\dim \rho)^{\omega}$.
For the last sentence of the theorem statement, we have $\sum_{\rho \in R_{\textup{sep}}} d_{\rho}^\omega \leq \sum d_{\rho}^2 d_{max}^{\omega-2} = d_{max}^{\omega-2} \cdot D$.
\end{proof}

\begin{remark}
The image $\overline{\rho}(\C[G])$ is the full $d_\rho \times d_\rho$ matrix ring if and only if $\rho$ is an irreducible representation. In particular, although we will not take advantage of this in the present paper, we note that when $\rho$ is not irreducible, we can replace $\rk \langle d_\rho, d_\rho, d_\rho \rangle$ (or $d_\rho^\omega$) with the tensor rank of multiplying matrices in the image of $\overline{\rho}$, which may only be a subspace of all matrices. (This was true in the case of finite groups as well, but over $\C$ for finite groups we can use irreducible representations without loss of generality. For finite groups and representations in characteristic dividing $|G|$ this no longer holds, and for infinite groups even in characteristic zero it need not hold.)
\end{remark}

\subsection{Border rank version in Lie groups} \label{sec:border}
For this section, let $G$ be a (real or complex) Lie group. This includes familiar groups such as $\GL_n(\R)$, $\GL_n(\C)$, the orthogonal group $\Orth_n$, and the unitary group $\U_n$. Indeed, these examples will be the primary groups we use in our constructions later in the paper, although the framework laid out in this section is by no means limited to these particular examples.

For the purposes of this section, by a \emph{1-parameter family} of elements of $G$ we mean an analytic function $x \colon (0,\alpha) \to G$ for some $\alpha > 0$. If $X = \{x_1, \dotsc, x_k\}$ is a collection of 1-parameter families $x_i$ and $\varepsilon$ is in their domain of definition, then we write $X(\varepsilon) = \{x_1(\varepsilon), \dotsc, x_k(\varepsilon)\}$, which is just a finite subset of $G$.

\begin{definition}[Border-separating functions] \label{def:sep-border}
Given sets $X,Y,Z$ of 1-parameter families of elements of $G$ with domain of definition $(0,\alpha)$, a \emph{set of border-separating functions for $(X,Y,Z)$} is a collection of analytic functions $\{f_{x,z}\colon G \times (0,\alpha) \to \C \mid x \in X, z \in Z\}$ such that, for  $0 < \varepsilon < \alpha$, 
\[
f_{x,z}(g, \varepsilon) = \begin{cases} 1 + O(\varepsilon) & \text{if } g=x(\varepsilon) z(\varepsilon)^{-1}, \text{ and}\\
                 0 + O(\varepsilon) & \text{if } g \in X(\varepsilon)Y(\varepsilon)^{-1} Y(\varepsilon)Z(\varepsilon)^{-1}
                                       \setminus\{x(\varepsilon)z(\varepsilon)^{-1}\}.\end{cases}
\]
Here, as is standard for analytic functions of a small variable $\varepsilon$, the big-Oh notation means asymptotically as $\varepsilon \to 0$.
\end{definition}

Now, we take our representative functions to come from \emph{analytic} representations of our Lie group, and, as in most things related to border rank, allow ourselves 1-parameter families of representative functions. Let $\C^{(0,\alpha)}$ denote the set of analytic functions $(0,\alpha) \to \C$. Given a set $\mathcal{I}$ of analytic representations of $G$, we write $\Fun_{\textup{fam}}(\mathcal{I})$ for the $\C^{(0,\alpha)}$-linear span of all the representative functions associated to any $\rho \in \mathcal{I}$.

Given three sets $X,Y,Z$ of 1-parameter families of elements of $G$, we say they satisfy the TPP if $X(\varepsilon), Y(\varepsilon), Z(\varepsilon)$ satisfy the TPP for all $\varepsilon$ in their domain of definition. Given such $X,Y,Z$, we may encode $\varepsilon$-approximate matrix multiplication (\`a la border rank) into the group ring $\C[G]$ in a similar manner as before, but now parameterized by $\varepsilon \in (0,\alpha)$. For an $|X| \times |Y|$ matrix $A$ and $|Y| \times |Z|$ matrix $B$, we define the following functions $(0,\alpha) \to \C[G]$:
\[
\overline{A}(\varepsilon) = \sum_{i \in [n],j \in [m]} A[i,j] (x_i(\varepsilon)y_j(\varepsilon)^{-1}) \quad\textup{and}\quad \overline{B}(\varepsilon) =
  \sum_{j \in [m],k\in [\ell]} B[j,k](y_j(\varepsilon)z_k(\varepsilon)^{-1}).
\]
Then
\begin{equation} \label{eqn:group-algebra-border}
\overline{A}(\varepsilon)\cdot \overline{B}(\varepsilon) = \sum_{i \in [n],k \in [\ell]} (AB)[i,k](x(\varepsilon)z(\varepsilon)^{-1}) + E(\varepsilon),
\end{equation}
where $E(\varepsilon) \in \C[G]$ is supported on $X(\varepsilon)Y(\varepsilon)^{-1}Y(\varepsilon)Z(\varepsilon)^{-1} \setminus X(\varepsilon)Z(\varepsilon)^{-1}$.

\begin{theorem} \label{thm:framework-border}
Let $G$ be a \textbf{Lie}\footnote{We have put in \textbf{bold} the parts of the statement of Theorem~\ref{thm:framework-border} that differ from Theorem~\ref{thm:framework}.} group, with finite sets  $X,Y,Z$ of \textbf{1-parameter families} satisfying the TPP. If $R_{\textup{sep}}$ is a finite set of finite-dimensional \textbf{analytic} representations of $G$ such that $\Fun_{\textbf{fam}}(R_{\textup{sep}})$ contains a set of \textbf{border-}separating functions for $(X,Y,Z)$, then 
\[
\left(|X| \, |Y| \, |Z| \right)^{\omega/3} \leq \sum_{\rho \in R_{\textup{sep}}} (\dim \rho)^\omega.
\]
In particular, for $D = \sum_{\rho \in R_{\textup{sep}}} (\dim \rho)^2$, and $d_{max} = \max\{\dim \rho : \rho \in R_{\textup{sep}}\}$,
\[
\left(|X| \, |Y| \, |Z| \right)^{\omega/3}\leq D \cdot d_{max}^{\omega-2}.
\]
\end{theorem}

The proof is very similar to Theorem~\ref{thm:framework}, but using border rank instead of rank, so we postpone it to Appendix~\ref{sec:proof:border}, where we state only the parts where we need to highlight differences from the previous proof. 

An appealing way to use the freedom of border rank is to first find a TPP construction of Lie subgroups $X,Y,Z$, and then to define finite subsets of those three using their Lie algebras, viz.:
\[
X' = \{\exp(\varepsilon a): a \in A, \text{ some finite subset of the Lie algebra of $X$}\},
\]
where $\varepsilon \to 0$ is the parameter we use for border rank. We will in fact show how to do this in a fairly generic way in Lemma~\ref{lem:split} below, using some additional machinery that we develop first.

\subsection{From representations to degree bounds in $\GL_n$ and $\U_n$}
\label{sec:irrepbound}

In this paper, our constructions will all take place in $\GL_n$ or (slight variants of) the unitary group $\U_n$, although the framework of Theorems~\ref{thm:framework} and \ref{thm:framework-border} is much more general. These groups have some useful properties that will motivate some of the machinery we develop below---even though that machinery will also end up being more general---so we take a brief interlude to highlight the relevant properties of $\GL_n$, before returning to the general abstract framework in the next section.

For both $\GL_n$ and $\U_n$, there is a natural correspondence between representations and polynomials of a given degree. By focusing on separating polynomials (instead of more general separating functions), this correspondence will allow our constructions to focus only on the \emph{degree} of the separating polynomials. The upshot is that instead of thinking directly about what representations will comprise $R_{\textup{sep}}$, we can focus solely on the degree of our separating polynomials when working in these groups.

In the rest of this section we review the relevant aspects of the representation theory of $\GL_n$ and $\U_n$, and extract from those the relevant target bounds on degree to get bounds on $\omega$. A standard reference for the representation theory of these groups is \cite{fulton1991representation}. It is a standard fact in representation theory that the finite-dimensional representation theory of these two groups are essentially the same, and everything we say in this section will apply to both groups.

The irreducible polynomial representations $\rho$ of $\GL_n(\C)$ and $\U_n$---where the entries of the matrix $\rho(g)$ are polynomials in the entries of the matrix $g$---are indexed by integer partitions $\lambda = (\lambda_1, \ldots, \lambda_{n})$ with at most $n$ parts, where $\lambda_1 \ge \lambda_2 \ge \cdots \ge \lambda_{n} \ge 0$. The matrix coefficients of the irreducible representations indexed by partitions of $s$ span
the functions from $G$ to $\C$ that are expressible as degree $s$
polynomials in the entries of $\GL_n$. We write $\irr_s(G)$ to denote the set of (pairwise non-isomorphic) irreducible representations of $G$ indexed by partitions of the integer $s$. 

The following bound will determine the target degree of the separating polynomials in our constructions, as in Corollary~\ref{cor:irrepbound} below.

\begin{lemma}\label{lem:maxdim}
    Let $n \geq 3$ and $s \ge 2$. 
    For $\rho \in \Irr_s(\GL_n)$ (resp., $\Irr_s(\U_n)$), $\dim(\rho) \le s^{\binom{n}{2}}$.
\end{lemma}

Although the result follows from some simple asymptotic analysis of standard results about the representation theory of $\GL_n$, we could not find it in the literature, so we provide a full proof in Appendix~\ref{sec:proof:maxdim}.

Here and below it is helpful to think of $n$ as large
but fixed, and $s$ growing independently of $n$.

\begin{corollary} \label{cor:irrepbound}
Suppose $X,Y,Z \subseteq \GL_n(\C)$ (or $\U_n$) satisfy the TPP, and there is a set of separating polynomials for $(X,Y,Z)$ of degree at most $s$.  Then $\omega$ satisfies
\[
(|X| \, |Y| \, |Z|)^{\omega / 3} \leq s^{\binom{n}{2}(\omega - 2)} \cdot \binom{s+n^2}{n^2}.
\]
In particular, if $|X|, |Y|, |Z| \geq s^{(1-o_s(1))(n^2/2 - o_n(n))}$, then $\omega=2$.

The same holds if $X,Y,Z$ are 1-parameter families that satisfy the TPP and we replace separating polynomials with border-separating polynomials.
\end{corollary}

Equivalently, if $|X|,|Y|,|Z| \geq q^{n^2/2 - o_n(n)}$ and there are (border-)separating polynomials of degree at most $q^{1 + o_q(1)}$, then $\omega=2$. We note that $|X|,|Y|,|Z|$ cannot have size $\geq s^{(1-o_s(1))(n^2/2 + c)}$ for any $c > 0$, as that would imply $\omega < 2$. Equivalently in terms of $q$, if the sets have size at least $q^{n^2/2 - o_n(n)}$, there cannot be separating polynomials of degree at most $q^{1-c}$ for any $c > 0$.

\begin{proof}
The space of polynomials of degree at most $s$ in the entries of a matrix in $\GL_n$ has dimension $\binom{s+n^2}{n^2}$, and the entries of $\Irr_i$ for $i \le s$ are a basis for this space \cite{DRS, DEP} (see also \cite{DKR,swan,procesi}). By applying Theorem \ref{thm:framework} (or Theorem~\ref{thm:framework-border} in the border rank case) with $R_{\textup{sep}} = \cup_{i=0}^s\Irr_i(\GL_n)$, and using the bound on the maximum dimension of an irreducible representation in this set given by Theorem \ref{lem:maxdim}, the first statement follows.

This then implies the bound
\[ \omega \le \frac{\log_s \binom{s+n^2}{n^2}  - 2 \binom{n}{2}}{\frac{1}{3} \log_s |X|\,|Y|\,|Z| - \binom{n}{2}} \le \frac{n^2 \log_s (s+n^2)  - n^2 + n}{1/3\log_s |X|\,|Y|\,|Z| - n^2/2 + n/2}.\]
Assuming that $|X|, |Y|, |Z| \geq  s^{(1-o_s(1))(n^2/2 - g(n))}$ for some $g(n) \in o_n(n)$, this gives
\[
\omega  < \frac{n^2 (\log_s (s+n^2)  - 1) + n}{(1-o_s(1))(n^2/2 - g(n)) - n^2/2 + n/2} \leq \frac{n + n^2 \cdot o_s(1) }{(n/2-g(n)) + o_s(1)(g(n) - n^2/2)}
\]
Now, choose $n$ such that $n/2 - g(n) > n/(2+\delta)$ with $\delta>0$ small. Then, taking the limit as $s \to \infty$, we get the bound $\omega \leq 2 + \delta$.
\end{proof}

Next we show that the Lie TPP construction of \cite{cu2003} (Theorem~\ref{thm:running-ex-1} above) in fact admits separating polynomials of degree $O(q^2)$. In Section~\ref{sec:invariant}, we will show how to use a border-rank version of our framework to construct \emph{finite} sets $X,Y,Z$, whose size we can then compare to the degree of the separating polynomials as in Corollary~\ref{cor:irrepbound}.

\bigskip
\hrule
\vspace{-.3in}
\begin{adjustwidth}{.2in}{.2in}
\subsection*{\center Separating polynomials for the running example in $\GL_n(\R)$}

We begin by describing a finite subset $Y_q$ of the orthogonal group $\Orth_n$, with the property that the diagonal entries of matrices in the quotient set $Q(Y_q)$ have a limited number of possible values, in the sense formalized in the following lemma. Here we treat $n$ as fixed, so the implicit constant in $O(q^2)$ can depend on $n$.

\begin{lemma}
For each positive integer $q$, there exists a subset $Y_q \subseteq \Orth_n$ of cardinality at least $q^{n^2/2 - (5/2)n}$, with the following property: for all $y, y' \in Y_q$, if $y$ and $y'$ agree in their first $i$ columns, then  
\[(y^Ty')[i+1, i+1] \in W_q,\]
where $W_q \subseteq \R$ is a fixed finite set of cardinality $O(q^2)$.
\label{lem:special-set-of-orthogonal-matrices}
\end{lemma}

\begin{proof}
For each $i$, let $V_i \subseteq \R^i$ be a subset of $q^{i-2}$ unit vectors with $O(iq^2)$ distinct
pairwise inner products. One can construct such a set by taking integer vectors with entries in $[O(q)]$, restricting to the most-popular length (losing a multiplicative factor of $O(iq^2)$), and then normalizing to length $1$. The inner product of two vectors in $V_i$ takes on only $O(iq^2)$ different values, and we take $W_q$ to be the union of these sets of values for $1 \le i \le n$.

Our set $Y_q$ consists of all orthogonal matrices constructed by selecting successive columns as follows: for the first column, pick an element of $V_n$; for the second column, pick an element of $V_{n-1}$ under an isometry from $\R^{n-1}$ to the orthogonal
complement of the first column; for the third column, pick an element of $V_{n-2}$ under an isometry from $\R^{n-2}$ to the orthogonal complement of the first two columns, and so
on. In each case, our choice of isometry depends only on the previous columns.
The number of matrices formed this way is $q^{(n^2-n)/2 - 2n}$, which is the promised cardinality of $Y_q$.

Now for $y, y' \in Y_q$, if they agree in their first $i$ columns, then the $(i+1)$-st columns of $y$ and $y'$ were selected from the set $V_{n-(i+1)}$ with the \emph{same} isometry applied to it (and an isometry does not change the pairwise inner products).  Thus
$(y^Ty')[i+1, i+1] \in W_q$, as required. 
\end{proof}

We remark that one might hope to improve the cardinality of $W_q$ to $O(q^{2-\varepsilon})$ by replacing $V_i$ with a more cleverly chosen set of unit vectors with fewer distinct pairwise inner products. This would lead to an improved bound on the degree of the separating polynomials described in the following theorem, and if one could take $\varepsilon = 1$ this would yield separating polynomials of optimal degree. Unfortunately, this is impossible by the solution to the high-dimensional version of the Erd\H{o}s distinct distances problem \cite{solymosi2008near}.

\begin{theorem}
Let $G,X,Y,Z$ be as in Theorem \ref{thm:running-ex-1}, let $X_q \subseteq X$ be those lower unitriangular matrices whose below-diagonal entries are integers in $\{1, 2, \ldots, q\}$,  let $Z_q \subseteq Z$ be those upper unitriangular matrices whose above-diagonal entries are integers in $\{1, 2, \ldots, q\}$, and let $Y_q \subseteq Y$ be the set of orthogonal matrices guaranteed by Lemma \ref{lem:special-set-of-orthogonal-matrices}. Then there are separating polynomials for $X_q, Y_q, Z_q$ of degree $O(q^2).$
\end{theorem}

\begin{proof}
Let $x,x' \in X_q$, $y, y' \in Y_q$, and $z,z' \in Z_q$, and let $M = x'y^Ty'z'$. We must devise a polynomial $p_{x,z}$ for which $p_{x,z}(M) = 1$ when $x = x'$, $y = y'$, and $z = z'$, and $p_{x,z}(M) = 0$ otherwise. 

Let $r$ be the univariate polynomial of degree $|W_q|-1$ for which $r(1) = 1$ and $r$ vanishes on the rest of $W_q$.  Let $s$ be a multivariate polynomial for which $s(M) = 1$ if $M$ agrees with $x$ in its first column and $s(M) = 0$ if $M$ takes on any other values in $\{1,2, \ldots, q\}$ in any of the entries of the first column of $M$. Observe that $s$ can be taken to have degree $O(q)$, where the implicit constant depends on $n$. Similarly, let $t$ be a multivariate polynomial for which $r(M) = 1$ if $M$ agrees with $z$ in its top row and $t(M)=0$ if $M$ takes on any other values in $\{1,2, \ldots, q\}$ in any of the entries of the top row of $M$. Again $t$ can be taken to have degree $O(q)$. 

Now consider the polynomial
\[r(M[1,1])s(M)t(M).\]
As noted in the proof of Theorem \ref{thm:running-ex-1}, the top-left entry of $M$ equals the top-left entry of $y^Ty'$. So if the $r(M[1,1])$ factor above does not vanish, then $y$ and $y'$ agree in their first columns, and at that point we know that the left column of $M$ agrees with the left column of $x'$ and the top row of $M$ agrees with the top row of $z'$. So if the $s(M)$ and $t(M)$ factors don't vanish, we know that $x$ and $x'$ agree in their first column, and $z$ and $z'$ agree in their top row. 

Now consider multiplying $M$ on the left by the lower unitriangular matrix $x_0$ that agrees with $x$ in the first column and has zeros elsewhere below the diagonal, and on the right by the upper unitriangular matrix $z_0$ that agrees with $z$ in the top row and has zeros elsewhere above the diagonal. Then the lower-right $(n-1) \times (n-1)$ submatrix of $M'$ is of the form $x''y''z''$ where $x''$ is the lower-right $(n-1) \times (n-1)$  submatrix of $x'$, $z''$ is the lower-right $(n-1) \times (n-1)$  submatrix of $z'$, and $y''$ is the lower-right $(n-1) \times (n-1)$  submatrix of $y^Ty'$. We can then construct a polynomial with three factors as above that when applied to $x_0Mz_0$ vanishes unless $y$ and $y'$ agree in their second column, and $x'$ agrees with $x$ in its second column, and $z'$ agrees with $z$ in its second-from-top row.

Continuing in this fashion, we obtain a sequence of $n$ such polynomials and take $p_{x,z}$ to be their product.
The degree of $p_{x,z}$ is $O(q^2)$ and notice that this comes only from the $r$ factors; the $s$ and $t$ factors always have degree $O(q)$.
\end{proof} 

Note that here we have $\dim X = \dim Y = \dim Z = (n^2 - n)/2$, but for Corollary~\ref{cor:irrepbound} we would need them to have dimension $n^2 / 2 - o_n(n)$, and the separating polynomial we get has degree $O_q(q^2)$, but Corollary~\ref{cor:irrepbound} needs degree at most $q^{1 + o_q(1)}$. In Section~\ref{sec:construction} we give a construction of the same relative dimension (namely, $\frac{\dim G}{2} - \Theta(n)$) but with separating polynomials of degree $O(q)$, meeting the degree bound of Corollary~\ref{cor:irrepbound}.
\end{adjustwidth}
\hrule
\bigskip

\subsection{Separating polynomials from a single invariant polynomial} \label{sec:invariant}
In this section, we return to the case of a general matrix Lie group $G \subseteq \GL_n(\C)$ or $\GL_n(\R)$, and we show how to leverage invariant theory to construct separating polynomials whose degree is not too large, and that thus have a hope of meeting the degree bound of Corollary~\ref{cor:irrepbound} in the case of $G = \GL_n(\C)$ or $G = \U_n$. 

The key result in this section, Lemma~\ref{lem:split}, combines border rank and Lie algebras, as suggested at the end of Section~\ref{sec:border}, with a new ingredient from invariant theory. The lemma lets us ``split'' the construction of sets satisfying the TPP and (border-)separating polynomials into two parts: the first part only involves the first and third sets $X$ and $Z$ (essentially, a separating polynomial version of the ``double product property'' (DPP), which is that $x^{-1} x' z^{-1} z' = 1$ iff $x=x'$ and $z=z'$)), and the second part involves only the middle set $Y$ and the ring of $XZ$-invariant polynomials. From one view, this lets us start with a $Y$, and ``work backwards,'' shifting the design task onto $X$, $Z$, and their invariant polynomials.

More specifically, Lemma~\ref{lem:split} instantiates the following idea. Roughly, we would like to construct the separating polynomial $f_{x,z}$ as a product of three factors: one that is the indicator function of $x$ among the set $X$ (1 on $x$, zero on all other elements of $X$), one that is the indicator function of $z^{-1}$ among $Z^{-1}$, and one on that is the indicator function $p_0$ of $I$ among $Y^{-1} Y$. However, the issue here is that $f_{x,z}$, and hence these factors, does not really have separate access to these three parts: it only gets as input the product $xy^{-1}y'z^{-1}$. Our initial motivation to use invariant polynomials was that if $p_0$ has the property that $p_0(xy^{-1}y'z^{-1}) = p_0(y^{-1}y')$ for all $x \in X, z \in Z$, then it in fact \emph{does} get direct access to only the $Y$-part of the input. Lemma~\ref{lem:split} then gives generic machinery that, from such an $XZ$-invariant $p_0$, constructs a full set of (border-)separating polynomials, which have the correct degree if $p_0$ has the correct degree.

One advantage of this approach, in addition to separating out the design task associated with $Y$ more from the specification of $X$ and $Z$, is that it only necessitates the construction of at most three polynomials $p_0, p_X, p_Z$, rather than a whole set $\{f_{x,z}\}$ of $|X|\,|Z|$-many polynomials, and the $p_X, p_Z$ polynomials are usually easy to construct if $X$ and $Z$ satisfy the DPP.

We now proceed with the technical details. By a \emph{polynomial} on a matrix Lie group $G \subseteq \GL_n(\C)$ or $G \subseteq \GL_n(\R)$, we mean a function $G \to \C$ that can be expressed as a polynomial in the $n^2$ matrix entries $x_{ij}$ of elements of $G$, with complex coefficients,\footnote{Warning: for readers familiar with algebraic geometry, if $G$ is an algebraic group our definition here is \emph{not} quite the same thing as an element of the coordinate ring of $G$. For example, when $G = \U_n$, we view $G$ as a subset of $\GL_n(\C)$ and want to consider polynomials only in the $n^2$ matrix entries of these complex matrices, but as an algebraic group $G$ is a \emph{real} algebraic group, with twice as many (real) coordinates, and this difference of a factor of 2 can actually be important. We believe there is a modification of our framework that works in the setting of algebraic groups, but leave that for future work.} and similarly for a polynomial on the Lie algebra $\Lie(G) \subseteq M_n(\C)$. By a 1-parameter family of polynomials we mean a polynomial in the $n^2$ matrix entries $x_{ij}$ of elements of $G$, with coefficients in $\C^{(0,\alpha)}$ for some $\alpha > 0$. 

Because we will use it frequently, we introduce the notation
\[
\Inv_{X,Z} = \{p \in \C[M_{ij} : i,j \in [n]] \mid p(xMz) = p(M) \text{ for all } x \in X, z \in Z\}.
\]
It is readily observed that $\Inv_{X,Z}$ is closed under multiplication and addition, and hence is a subring of the polynomial ring. In favorable (and many well-studied) cases, this subring is finitely generated over $\C$ (see, e.\,g., \cite{GoodmanWallach, Dolgachev}), and when a finite generating set of $\Inv_{X,Z}$ is known, we can focus on designing a polynomial that is composed with those generating invariants---which is then automatically invariant itself---in order to obtain $p_0$. We will see an example of this below.

In a Lie group $G$, for convenience let us say a subset $X \subseteq G$ is a \emph{Lie submanifold} if $X$ is a submanifold containing the identity of $G$, and such that for every $v$ in the tangent space of $X$ at the identity, for all sufficiently small $\varepsilon > 0$, the exponential $\exp(\varepsilon t)$ is in $X$. If $X$ is a Lie submanifold,  we use $\Lie(X)$ to denote the tangent space to $X$ at the identity, even though this need not be a Lie algebra.

\begin{lemma}[Splitting separating functions into invariant functions and double-product property] \label{lem:split}
Let $G$ be a matrix Lie subgroup of $\GL_n(\F)$ for $\F \in \{\R, \C\}$, with Lie submanifolds $X,Z$ and a finite set $Y$ of 1-parameter families. Let $q$ be a positive integer.

If there exist
\begin{itemize}
\item continuous functions $f_X\colon \F^{d_X} \to \Lie(X)$ and $f_Z\colon \F^{d_Z} \to \Lie(Z)$, and polynomials $p_X\colon \Lie(G) \to \F^{d_X}$ and $p_Z \colon \Lie(G) \to \F^{d_Z}$ such that
\[
p_X(f_X(v) - f_Z(v')) = v \qquad \text{ and } \qquad p_Z(f_X(v) - f_Z(v')) = v',
\]
for all $v \in \F^{d_X},v' \in \F^{d_Z}$,\footnote{Notice that, since $p_X, p_Z$ essentially invert $f_X, f_Z$ on the sum of their images, this condition implies that $\Span_{\F}(\Image(f_X)) \cap \Span_{\F}(\Image(f_Z)) = 0$, a DPP-like condition.} 

and
\item a 1-parameter family of polynomials $p_0(g,\varepsilon) \in \Inv_{X,Z}$ such that
\[
p_0(g,\varepsilon) = \begin{cases}
1 + O(\varepsilon) & \text{if } g = I \\
0 + O(\varepsilon) & \text{if } g \in Y(\varepsilon)^{-1}Y(\varepsilon) \setminus \{I\},
\end{cases}
\]
\end{itemize}
then there exist
\begin{itemize}
\item a reparametrization $Y'$ of $Y$ and finite sets $X',Z'$ of 1-parameter families in $X$, $Z$ (resp.) such that $X'(\varepsilon), Y'(\varepsilon), Z'(\varepsilon)$ satisfy the TPP for all $\varepsilon$ in some non-empty range $(0,\alpha')$, $|X'| \geq q^{d_X}$, and $|Z'| \geq q^{d_Z}$, and 

\item a set of border-separating polynomials for $(X',Y',Z')$ of degree\footnote{Degree as measured only as a function of the matrix entries---$\varepsilon$ is still considered a separate parameter. Equivalently, degree as polynomials with coefficients in $\C^{(0,\alpha)}$.} at most $\deg p_0 + O((d_X+d_Z) q)$.
\end{itemize}
\end{lemma}

As the first condition of Lemma~\ref{lem:split} may be a little intimidating, before the proof we give an example to help allay such fears:

\begin{lemma} \label{lem:disjoint}
If $X,Z \leq G \leq \GL_n(\R)$ are Lie subgroups that intersect trivially, then there exist $f_X, f_Z, p_X, p_Z$ satisfying the first condition of Lemma~\ref{lem:split}, with $d_X = \dim X$ and $d_Z = \dim Z$.
\end{lemma}

\begin{proof}
If $X,Z$ intersect trivially, then $\Lie(X) \cap \Lie(Z) = 0$: otherwise, taking any nonzero $v \in \Lie(X) \cap \Lie(Z)$, the exponential $\exp(\varepsilon v)$ would be in $X \cap Z$ for any sufficiently small $\varepsilon > 0$. Let $f_X \colon \R^{\dim X} \to \Lie(X)$ denote any linear isomorphism of vector spaces, and similarly for $f_Z$. As $\Lie(X) \cap \Lie(Z) = 0$, any $v \in \Lie(X) + \Lie(Z) = \Image(f_X) + \Image(f_Z)$ has a unique expression as $a + b$ with $a \in \Lie(X)$ and $b \in \Lie(Z)$. We define $p_X$ as a linear function on $\Lie(X) + \Lie(Z)$, and it may then be extended to all of $\Lie(G)$ arbitrarily (e.\,g., by taking a linear basis of $\Lie(G)$ that includes as a subbasis a basis for $\Lie(X) + \Lie(Z)$, and defining $p_X=0$ on basis elements outside of $\Lie(X) + \Lie(Z)$). On $\Lie(X) + \Lie(Z)$, since $\Lie(X) \cap \Lie(Z) = 0$, we may take $p_X$ to be the natural projection onto $\Lie(X)$ with kernel $\Lie(Z)$, followed by $f_X^{-1}$ (recall that we defined $f_X$ to be invertible on $\Lie(X)$). We may define $p_Z$ similarly. The verification of the first condition of Lemma~\ref{lem:split} is then routine.
\end{proof}

Note that if we drop the requirement that $p_0$ be in $\Inv_{X,Z}$, then the hypotheses of Lemma~\ref{lem:split} would be almost trivially satisfied whenever $X$ and $Z$ have trivial intersection. In applying Lemma~\ref{lem:split}, it is only the use of invariants that adds any real difficulty; however, the use of invariants is also a lynchpin to reach the conclusion.

\begin{proof}[Proof of Lemma~\ref{lem:split}]
Let $G$, $X$, $Y$, $Z$, $p_0$, $d_X$, $d_Z$, $f_X$, $f_Z$, $p_X$, and $p_Z$ be as in the statement. Without loss of generality, we may assume that the 1-parameter families in $Y$ use only nonnegative powers of $\varepsilon$, by multiplying them by a large power of $\varepsilon$ and dividing each monomial of $p_0$ by the corresponding power of $\varepsilon$.

Let $A, B \subseteq \F$ be sets of cardinality $q$. Define
\[
X' = \{\exp(\varepsilon f_X(a)) \mid a \in A^{d_X}\} \quad\text{and}\quad Z' = \{\exp(\varepsilon f_Z(b)) \mid b \in B^{d_Z}\},
\]
where here ``$\exp(\varepsilon x)$'' denotes the function $\varepsilon \mapsto \exp(\varepsilon x)$; that is, we remind the reader that each element of $X'$ is really a 1-parameter family. In doing this, we may need to restrict the domain of definition of our 1-parameter families, but it will still be of the form $(0,\alpha')$ for some $\alpha' > 0$, since we are introducing only finitely many new 1-parameter families. 

We show here that $X'$ and $Z'$ satisfy the double product property, as we will use this below to show the border-separating family. Suppose $x_1(\varepsilon) z_1(\varepsilon)^{-1} = x_2(\varepsilon) z_2(\varepsilon)^{-1}$ for all sufficiently small $\varepsilon > 0$. Writing these out as power series in $\varepsilon$, we have 
\[
(1 + \varepsilon f_X(a_1) + O(\varepsilon^2)) (1 - \varepsilon f_Z(b_1) + O(\varepsilon^2)) = (1 + \varepsilon f_X(a_2) + O(\varepsilon^2)) (1 - \varepsilon f_Z(b_2) + O(\varepsilon^2)),
\]
for the appropriate values of $a_1, a_2 \in A^{d_X}$ and $b_1, b_2 \in B^{d_Z}$. Multiplying out, we see that the $\varepsilon$ terms on the two sides are $f_X(a_1) - f_Z(b_1)$ and $f_X(a_2) - f_Z(b_2)$, respectively, and since this holds for all sufficiently small $\varepsilon$, we in fact have $f_X(a_1) - f_Z(b_1) = f_X(a_2) - f_Z(b_2)$ as elements of $\Lie(G)$. Applying $p_X$ to both sides of the latter equality, we get 
\[
a_1 = p_X(f_X(a_1) - f_Z(b_1)) = p_X(f_X(a_2) - f_Z(b_2)) = a_2,
\]
and it follows that $b_1 = b_2$ as well. But since $x_i(\varepsilon) = \exp(\varepsilon f_X(a_i))$, we then get $x_1(\varepsilon) = x_2(\varepsilon)$ and $z_1(\varepsilon) = z_2(\varepsilon)$, thus showing the DPP for all sufficiently small $\varepsilon$.

Our border-separating polynomials will be a product of two factors: a reparametrization of $p_0$ and a function we denote $r_{a,b}$, which uses the above ideas. We will first define these two factors, then show how their product gives us a set of border-separating polynomials.

For $a \in A^{d_X}$ and $b \in B^{d_Z}$, we will define a polynomial $r_{a,b}$ on $\Lie(G)$, with complex coefficients, such that 
\begin{equation} \label{eqn:rab}
r_{a,b}(f_X(a') - f_Z(b')) = \begin{cases} 1 & \text{if } a'=a \text{ and } b'=b, \text{ and} \\
0 & \text{if } a' \in A^{d_X}, \ b' \in B^{d_Z}, \ a'-b' \neq a-b.
\end{cases}
\end{equation}
To do so, for $x \in A$ let $\alpha_x$ be a univariate polynomial on $\F$ that is $1$ at $x$ and $0$ on all other values in $A$; similarly for $y \in B$, let $\beta_y$ be $1$ at $y$ and $0$ at all other values of $B$. By Lagrange interpolation, $\alpha_x, \beta_y$ can be taken to have degree less than $|A|=|B|=q$. We will use crucially that $p_X(f_X(a') - f_Z(b')) = a'$ and $p_Z(f_X(a') - f_Z(b')) = b'$. For $a = (a_1, \dotsc, a_{d_X})$ and $b = (b_1, \dotsc, b_{d_Z})$, we define $r_{a,b}$ as
\[
r_{a,b}(v) = \left(\prod_{i=1}^{d_X} \alpha_{a_i}(p_X(v)_i) \right)\left(\prod_{i=1}^{d_Z} \beta_{b_i}(p_Z(v)_i) \right).
\]
It is then readily verified that this definition has the property \eqref{eqn:rab}, and the degree of $r_{a,b}$ is at most $2q(d_X + d_Z)(\deg p_X + \deg p_Z) \leq O((d_X+d_Z)q)$, since the degrees of $p_X$ and $p_Z$ are independent of $q$, by definition.

Next we define the other factor of our border-separating polynomial.\footnote{If we know that the families in $Y'$ have their constant term being the identity, that is, if every family in $Y'$ is of the form $I + O(\varepsilon)$, then we do not need to reparametrize and can simply use $p_0$ for this factor.} We reparametrize $p_0$ and the 1-parameter families in $Y$ by replacing $\varepsilon$ with $\varepsilon^t$ for some $t > \deg r_{a,b}$. That is, let us define $p_0'(g,\varepsilon) := p_0(g,\varepsilon^t)$ and $Y'(\varepsilon) := Y(\varepsilon^t)$. Then we have
\begin{align*}
p_0'(g,\varepsilon) = p_0(g,\varepsilon^t) & = 
\begin{cases}
1 + O(\varepsilon^t) & \text{if } g = I \\
0 + O(\varepsilon^t) & \text{if } g \in Y(\varepsilon^t)^{-1}Y(\varepsilon^t) \setminus \{I\},
\end{cases} \\
& = \begin{cases}
1 + O(\varepsilon^t) & \text{if } g = I \\
0 + O(\varepsilon^t) & \text{if } g \in Y'(\varepsilon)^{-1}Y'(\varepsilon) \setminus \{I\},
\end{cases}.
\end{align*}

Finally, for $x = \exp(\varepsilon f_X(a))$ and $z = \exp(\varepsilon f_Z(b))$, we define our border-separating polynomial
\[
p_{x,z}(M,\varepsilon) := p_0'(M,\varepsilon) r_{a,b}((M-I)/\varepsilon).
\]
It is clear that $\deg p_{x,z} = \deg p_0 + \deg r_{a,b} = \deg p_0 + O((d_X+d_Z)q)$. It remains to show that these $p_{x,z}$ indeed give a set of border-separating polynomials for $X',Y',Z'$.

Suppose $a' \in A^{d_X}$, $b' \in B^{d_Z}$, and $y_1, y_2 \in Y'$, and let
\[
M(\varepsilon)= \exp(\varepsilon f_X(a')) y_1(\varepsilon)^{-1} y_2(\varepsilon) \exp(\varepsilon f_Z(b'))^{-1}
\]
be an arbitrary element of $X'(\varepsilon)Y'(\varepsilon)^{-1}Y'(\varepsilon) Z'(\varepsilon)^{-1}$.
Then we have 
\begin{align*}
p_{x,z}(M(\varepsilon),\varepsilon) & = p_0'(M(\varepsilon), \varepsilon) r_{a,b}((M(\varepsilon) - I)/\varepsilon) \\
& = p_0'(y_1(\varepsilon)^{-1} y_2(\varepsilon), \varepsilon) r_{a,b}((M(\varepsilon) - I)/\varepsilon) \qquad \text{(since $p_0' \in \text{Inv}_{X,Z}$)}.
\end{align*}

If $y_1(\varepsilon)^{-1} y_2(\varepsilon) \neq I$, then we have $p_0'(y_1(\varepsilon)^{-1} y_2(\varepsilon), \varepsilon) = O(\varepsilon^t)$. In this case, since $t > \deg r_{a,b}$, even if $r_{a,b}(\dotsb/\varepsilon)$ contributes negative powers of $\varepsilon$, they are not enough to cancel out the $\varepsilon^t$ coming from $p_0'$, and we are left with $p_{x,z}(M(\varepsilon),\varepsilon) = O(\varepsilon)$, as desired.

Otherwise, we have $y_1(\varepsilon)=y_2(\varepsilon)$, and which implies $p_0'(M(\varepsilon),\varepsilon) = 1 + O(\varepsilon^t)$ by our assumption on $p_0$ and the construction of $p_0'$. Then, since $X'$ and $Z'$ satisfy the double product property, we have $M = x(\varepsilon) z(\varepsilon)^{-1}$ iff $a=a'$ and $b=b'$. And in this case, the above expression becomes 
\begin{align*}
 p_{x,z}(M(\varepsilon),\varepsilon) & = p_0'(I, \varepsilon) r_{a,b}(((I + \varepsilon f_X(a) + O(\varepsilon^2))(I - \varepsilon f_Z(b) + O(\varepsilon^2)) - I)/\varepsilon) \\
 & = (1 + O(\varepsilon^t)) r_{a,b}(f_X(a) - f_Z(b) + O(\varepsilon)) \\
 & = (1 + O(\varepsilon^t)) (r_{a,b}(f_X(a) - f_Z(b)) + O(\varepsilon)) \qquad \text{(since $r_{a,b}$ is a polynomial over $\F$)} \\
 & = (1 + O(\varepsilon^t)) (1 + O(\varepsilon)) \\
 & = 1 + O(\varepsilon),
\end{align*}
as desired.

On the other hand, if $a \neq a'$ or $b \neq b'$, then $r_{a,b}(f_X(a') - f_Z(b'))=0$, so $r_{a,b}(f_X(a') - f_Z(b') + O(\varepsilon))$ becomes $0 + O(\varepsilon)$, since $r$ is a polynomial with coefficients in $\F$, independent of $\varepsilon$. Then the product of $p_0'$ and $r_{a,b}$ is $O(\varepsilon)$ as well, again as desired. Thus, $\{p_{x,z} \mid x \in X', z \in Z'\}$ is indeed a set of border-separating polynomials for $(X', Y', Z')$, of degree $\deg p_0 + O((d_X + d_Z)q)$.

Finally, we use pieces of the above argument to show that $X'(\varepsilon),Y'(\varepsilon),Z'(\varepsilon)$ satisfy the TPP. Suppose $x_1,x_2 \in X'$, $y_1,y_2 \in Y'$, and $z_1, z_2 \in Z'$ satisfy 
\[
x_1(\varepsilon)^{-1} x_2(\varepsilon) y_1(\varepsilon)^{-1} y_2(\varepsilon) z_1(\varepsilon)^{-1} z_2(\varepsilon)=1
\]
for all sufficiently small $\varepsilon > 0$. Equivalently, we have 
\[
x_2(\varepsilon) y_1(\varepsilon)^{-1} y_2(\varepsilon) z_1(\varepsilon)^{-1} = x_1(\varepsilon) z_2(\varepsilon)^{-1}. 
\]
Since $\{p_{x,z} \mid x \in X', z \in Z'\}$ is a set of border-separating polynomials for $X', Y', Z'$, we have 
\[
p_{x_1, z_2}(x_2(\varepsilon) y_1(\varepsilon)^{-1} y_2(\varepsilon) z_1(\varepsilon)^{-1}, \varepsilon)=p_{x_1, z_2}(x_1(\varepsilon) z_2(\varepsilon)^{-1}, \varepsilon) = 1 + O(\varepsilon). 
\]
If $y_1(\varepsilon) \neq y_2(\varepsilon)$, then above we argued this value would have to be $O(\varepsilon)$, a contradiction. So we have $y_1(\varepsilon) = y_2(\varepsilon)$. And then, since $X'$ and $Z'$ satisfy the DPP, we conclude $x_1(\varepsilon)=x_2(\varepsilon)$ and $z_1(\varepsilon) = z_2(\varepsilon)$. Thus $(X', Y', Z')$ satisfy the TPP.
\end{proof}

Continuing our running example, we now show how to satisfy the hypothesis of Lemma~\ref{lem:split}, and thus get a TPP construction using border rank, invariant polynomials, and Lie algebras. 

\bigskip
\hrule
\vspace{-.3in}
\begin{adjustwidth}{.2in}{.2in}
\subsection*{\center Example of a separating polynomial in the invariant ring}

We return to the example in $\GL_n(\R)$ described in Theorem \ref{thm:running-ex-1}. 
The following lemma is straightforward to check. Let $\lpm_i$ denote the polynomial that is the $i$-th leading principal minor of an $n \times n$ matrix; that is, $\lpm_i$ is the determinant of the upper-left $i \times i$ submatrix of an $n \times n$ matrix.

\begin{lemma}
Let $X, Z \subseteq \GL_n(\R)$ be the sets of lower unitriangular matrices and the upper unitriangular matrices, respectively. Then $\Inv_{X,Z}$ contains the leading principal minors. 
\end{lemma}

In fact, in this case $\Inv_{X,Z}$ is generated by the leading principal minors, but we will not need this stronger fact for our constructions.

\begin{theorem}
Let $G,X,Y,Z$ be as in Theorem~\ref{thm:running-ex-1}. For any $q > 0$, there is a set $\widetilde{Y}$ of 1-parameter families taking values in $Y$, of size $q^{\dim Y}$, as well as functions $f_X, f_Z$, polynomials $p_X, p_Z$, and a 1-parameter family of polynomials $p_0$ of degree $O(q^2)$ satisfying the hypotheses of Lemma~\ref{lem:split} with $d_X = \dim X$,  $d_Z = \dim Z$, and $\widetilde{Y}$ replacing $Y$.

Consequently, there is a TPP construction with sizes $q^{\dim X}, q^{\dim Y}, q^{\dim Z}$, resp., admitting border-separating polynomials of degree $O(q^2)$.
\end{theorem}

\begin{proof}
That the first condition of Lemma~\ref{lem:split} is satisfied follows immediately from Lemma~\ref{lem:disjoint}.

Our set $\widetilde{Y}$ will be the 1-parameter families
$\exp(\varepsilon A)$ where $A$ is skew-symmetric and with entries in $[-q/2, q/2] \cap \Z$. 

Let $M = y^{-1} y' = \exp(\varepsilon
A)\exp(\varepsilon B)^{-1}$ be an arbitrary element of $\widetilde{Y}^{-1} \widetilde{Y}$. Then
\[
M = I + \varepsilon(A-B) +
                   \frac{\varepsilon^2}{2}(A^2 - 2AB + B^2) +
        O(\varepsilon^3).
\]
A key observation is that the diagonal of $A^2 - 2AB + B^2$ agrees
with the diagonal of $(A-B)^2$. This is because the diagonal of $AB$
equals the diagonal of $(AB)^T = B^TA^T = BA$, since $A$ and $B$ are skew-symmetric. 

Let us consider the function $\lpm_j(M)$. There are only two
types of summands in the determinant that do not have $\varepsilon^3$
factors. One is the product along the main diagonal:
\[\prod_{i=1}^j \left (1 + \frac{\varepsilon^2}{2}(A-B)^2[i,i] +O(\varepsilon^3)\right )
  = 1 +  \frac{\varepsilon^2}{2}\sum_{i=1}^j(A-B)^2[i,i] +
  O(\varepsilon^3).\]
Observe that $A-B$ is skew symmetric, and so $(A-B)^2[i,i]$ equals
$-\sum_{i'=1}^n (A-B)[i,i']^2$ (the negative length-squared of the $i$-th row).
The other type of summand is one whose corresponding (odd) permutation is a
single transposition; these contribute
\[-\varepsilon^2\sum_{1 \le i < i' \le j} (A-B)[i,i'](A-B)[i',i]
  +O(\varepsilon^3) = - \frac{\varepsilon^2}{2}\sum_{1 \le i \ne i' \le j}(A-B)[i,i'](A-B)[i',i]
  +O(\varepsilon^3)\]
which equals $ \frac{\varepsilon^2}{2}\sum_{1 \le i,i' \le j}
 (A-B)[i,i']^2 +O(\varepsilon^3)$ by the skew-symmetry of
 $A-B$. Altogether, we have
 \[\lpm_j(M) = 1 - \frac{\varepsilon^2}{2} \sum_{\substack{1 \le i \le j\\ j
     < i' \le n}} (A-B)[i,i']^2.\]
Thus,
 \[\sum_{j=1}^n\lpm_j(M) = n - \frac{\varepsilon^2}{2} \sum_{1 \le i< i' \le n}
   (i'-i)(A-B)[i,i']^2 + O(\varepsilon^3).\]
Observe that the $\varepsilon^2$ term takes on $O(q^2)$ possible values, and let $r(z)$ be the univariate polynomial that is $1$ when $z = 0$ and 0 when $z$
is any other possible values for the $\varepsilon^2$ term. Then our separating polynomial is
\[p_0(M) = r\left (\left (-n - \sum_{j=1}^n\lpm_j(M)\right )/\varepsilon^2\right ),\]
which is $1+O(\varepsilon)$ if $A = B$ and $0 + O(\varepsilon)$
otherwise. As $r$ has degree $O(q^2)$ and $j \leq n \leq O_q(1)$, we have $\deg p_0 \leq O_q(q^2)$.
\end{proof}
\end{adjustwidth}
\hrule

\section{Separating polynomials of degree $O(q)$} \label{sec:construction}

In this section we give a construction, now in a non-compact special unitary group, in which the Lie subgroups $X,Y,Z$ satisfy the TPP and have dimension approaching half the ambient dimension.  We show the existence of finite sets $X_q \subseteq X$, $Y_q \subseteq Y$, and $Z_q \subseteq Z$ satisfying $|X_q| = q^{\dim{X}}$, $|Y_q| = q^{\dim{Y}}$, $|Z_q| = q^{\dim{Z}}$ and with separating polynomials of degree $O(q)$. As shown in Corollary~\ref{cor:irrepbound}, this degree bound is the quantitative requirement on the degree for achieving exponent 2; where the construction falls short is that it approaches half the dimension ``just barely'' too slowly, at a rate of $\Theta_n(n)$ rather than $o_n(n)$, and second, it is in a unitary group, so it approaches half the dimension of the unitary group, rather than half the dimension of $\GL_n(\C)$.

\subsection{The construction} \label{sec:construction-details}

We assume $n$ is even and define the containing group
\begin{align*}
G = \SU_{n/2,n/2} = \{M \in \GL_n(\C) \mid M^*QM = Q \text{ and} \det(M)=1\}, \quad \text{with } \ \
Q = \left( \begin{array}{c|c}I &  \\ \hline  & -I \end{array}\right).     
\end{align*}
Further, let  $J$ be the matrix with ones on the antidiagonal (in positions $(n+1-i,i)$),
\begin{align*}
U  & = \frac{1}{\sqrt{2}}\cdot \left( \begin{array}{c|c} I & J \\ \hline J & -I \end{array}\right),  \\
D_0 & = \diag(n, n-1, n-2, \ldots, n/2+1, (n/2+1)^{-1}, (n/2+2)^{-1}, \ldots, n^{-1}),
\end{align*}
and
\begin{align} \label{eq:D}
D = UD_0U^*.
\end{align}
The only properties we will need from these matrices are that $D_0$ is diagonal, $D_0^*D_0$ has distinct, positive, real diagonal entries, and $U$ is unitary such that $UD_0U^*$ is in $G$, but we use these particular matrices for concreteness. In Lemma~\ref{lem:D} we will show that $D$ as in \eqref{eq:D} is indeed in $G$.

The group $\SU_{n/2,n/2}$ is not compact, and it may be less familiar than the compact group $\SU_n$, but it is equally useful for our purposes. Both $\SU_{n/2,n/2}$ and $\SU_n$ are subgroups of $\SL_n(\C)$, and their irreducible representations are exactly the restrictions of those of $\SL_n(\C)$. See Appendix~\ref{sec:suvssl} for a brief review of this topic.

The subgroup $\U_n \cap G$ is the block diagonal subgroup
 $U_{n/2} \times U_{n/2}$, and we need the following subspace, which is close to $\Lie(U_n \cap G)$ but is just a bit smaller:
\[S = \left \{\left ( \begin{array}{c|c} A_0& 0 \\ \hline 0 & A_1 \end{array}\right ) \mid A_0, A_1 \text{ are skew-Hermitian with zeros on the diagonal}\right \}.\]

\begin{theorem} \label{thm:construction}
Let $G,$ $D$, and  $S$ be as above.
Let  $\mathcal{I}_q \subset \C$ be the set of $a+ib$ with $a,b \in [-\lceil\sqrt{q}/2 \rceil, \lceil \sqrt{q}/2\rceil]$.  Define the following $\varepsilon$-parametrized subsets of  $G$:
\begin{align*}
X &=  \{D^{-1}\exp(\varepsilon A)D \mid A \in S\}, \\
Z &=  \{D  \exp(\varepsilon A)D^{-1} \mid A \in S\}, \\
Y_q  &= \big\{ \exp(\varepsilon A) \mid A\in S, \ \text{entries of  $A$ lie in $\mathcal{I}_q$}
 \big\}.
\end{align*}
Then there exist  $f_X$, $f_Z$, $p_X$, $p_Z$, and $p_0$ of degree  $O(q)$ satisfying the hypothesis of Lemma \ref{lem:split}.
Hence by Lemma \ref{lem:split}, there exist three $\varepsilon$-parametrized subsets $X_q, Y_q, Z_q \subseteq G$, all of size at least $q^{n^2/4-n/2}$, which satisfy the TPP and admit border-separating polynomials of degree $O(q)$.
\end{theorem}

\subsection{A precursor construction in $\GL_n(\C)$}

In preparation for the proof of Theorem \ref{thm:construction}, we first establish a precursor in the containing group 
$\GL_n(\C)$.
Recall that the unitary group $\U_n$ is the set of matrices $M$ for which $M^*M = I$, and it has half the dimension of the containing group.\footnote{The unitary group is a real algebraic group, so we need to count real dimensions. The containing group has $2n^2$ real dimensions and the unitary group has $n^2$ real dimensions.} 
We will shortly show that three conjugates of $\U_n$ in $\GL_n(\C)$ almost satisfy the TPP (Theorem \ref{thm:three-conjugates-unitary}). The proof will make use of the following lemma:

\begin{lemma}
    If $M \in \U_n$ and $D = UD_0U^*$ as defined above, then
    \[\Tr(D^*M^*D^*DMD) \le \Tr((D^*D)^2)\]
    with equality if and only if $U^*MU$ is diagonal.
    \label{Lem:kvn-proof}
\end{lemma}

\begin{proof}
Substituting the definition of $D$, we have
\begin{align*}
    \Tr(D^*M^*D^*DMD) & =  \Tr(UD_0^*U^*M^*UD_0^*D_0U^*MUD_0U^*)\\
    & = \Tr(D_0^*M_0^*D_0^*D_0M_0D_0)
\end{align*}
with $M_0 = U^*MU$. Rewriting, we have
\[\Tr(D_0^*M_0^*D_0^*D_0M_0D_0) = \sum_{i,j} |D_0[i,i]|^2|D_0[j,j]|^2 |M_0[i,j]|^2.\]
Since $M_0$ is unitary, the matrix $M_0'$ whose $(i,j)$ entry is $|M_0[i,j]|^2$ is doubly stochastic. By the Birkhoff--von Neumann Theorem, $M_0'$ is then a convex combination of permutation matrices: $\sum_{\pi} \alpha_{\pi} \pi = M_0'$, where $\alpha_\pi \geq 0$ and $\sum \alpha_\pi = 1$. If we define the row vector $v$ by $v_i = |D_0[i,i]|^2$ then we can rewrite the above expression as
\[\sum_{i,j} |D_0[i,i]|^2|D_0[j,j]|^2 |M_0[i,j]|^2 = \sum_{\pi \in S_n} \alpha_\pi v \pi v^T.\]
Now $v$ has distinct entries, and so by the Cauchy--Schwarz inequality, $v\pi v^T \le v v^T$ with equality iff $\pi$ is the identity. We conclude that 
\[\Tr(D_0^*M_0^*D_0^*D_0M_0D_0) \le \Tr((D_0^*D_0)^2),\] with equality if and only if $M_0$ is diagonal (and hence has diagonal entries of norm $1$). Finally we note that $\Tr((D_0^*D_0)^2) = \Tr((D^*D)^2)$, which completes the proof.
\end{proof}

This lemma gives a convenient way to prove that the three subgroups defined in the following theorem satisfy the TPP up to a ``failure at the diagonal'': 

\begin{theorem}
Let $D = UD_0U^*$ be as defined as above. Then the three subgroups $X = D^{-1}\U_nD$, $Y = \U_n$, and $Z = D\U_nD^{-1}$ of $\GL_n(\C)$ satisfy the following property. For all \[x  = D^{-1}M_1D \in X, \quad y = M_2\in Y, \quad z =DM_3D^{-1}\in Z,\]
if $xyz = I$, then $U^*M_1U$, $U^*M_2U$, and $U^*M_3U$ are diagonal matrices. 
\label{thm:three-conjugates-unitary}
\end{theorem}

\begin{proof}
We have $I = xyz =(D^{-1}M_1D)(M_2)(DM_3D^{-1}) =: M$, where $M_1, M_2, M_3$ are unitary matrices.
By canceling various inverses, using that $M_1, M_3 \in \U_n$, and conjugating by $M_3$ (and recalling that trace is invariant under conjugation), we have
\[\Tr(D^*M^*D^*DMD) = \Tr(D^*M_2^*D^*DM_2D).\]
Combining with Lemma~\ref{Lem:kvn-proof}, we get
\[\Tr(D^*M^*D^*DMD) = \Tr(D^*M_2^*D^*DM_2D)\le \Tr((D^*D)^2),\]
with equality iff $U^*M_2U$ is diagonal. But since $M = I$ by assumption, $Tr(D^*M^*D^*DMD) = \Tr((D^*D)^2)$. So we know that $U^*M_2U$ is a diagonal matrix $\alpha$, which as a unitary matrix must have unit complex diagonal entries. 

Now we set $D_1 = DM_2D = UD_0^2\alpha U^*$ which, like $D$, is unitarily diagonalizable with eigenvalues having distinct norms, and consider 
\[\Tr(D^*M^*D^*D_1^{*}D_1DMD) = \Tr(D_1^*M_1D_1^{*}D_1M_1D_1) \le \Tr((D_1^*D_1)^2).\]
Since $M=I$, we obtain equality (to see why, write both sides in terms of $D_0$, $\alpha$, and $U$). Thus, by Lemma \ref{Lem:kvn-proof}, we find that $U^*M_1U$ is a diagonal matrix $\beta$, which as a unitary matrix must have unit complex diagonal entries. 

At this point we know that \[I = M = (D^{-1}U\beta U^*D)(U\alpha U^*)(DM_3D^{-1}),\] and observe that $D^{-1}U\beta U^*DU\alpha U^* = U\beta\alpha U^*$. Thus $DM_3D^{-1} =U\alpha^{-1}\beta^{-1}U^*$  which implies that $U^*M_3U$ must be diagonal.  
This concludes the proof of the theorem. 
\end{proof}

In the variant presented in the next section, we will form our TPP sets as exponentials of Lie algebra elements, and for this, we need a version of Lemma \ref{Lem:kvn-proof} that describes the deviation of $\Tr(D^*M^*D^*DMD)$ below $\Tr((D^*D)^2)$ explicitly as a function of the infinitesimal. Recall that the Lie algebra of $\U_n$ is the set of skew-Hermitian matrices. 

\begin{lemma}
Let $D = UD_0U^*$ be as defined above, with $d_1, d_2, \ldots, d_n$ being the diagonal entries of $D_0$. Let $A, B$ be skew-Hermitian matrices and let $M =
  \exp(\varepsilon A)(\exp(\varepsilon B))^{-1}$. 
Then
\[\Tr(D^*M^*D^*DMD) = \Tr((D^*D)^2) - \varepsilon^2\sum_{i<j}
  (|d_i|^2 - |d_j|^2)^2|C[i,j]|^2 + O(\varepsilon^3),\]
where $C = U^*(A-B)U$.
In particular, this quantity is $\Tr((D^*D)^2) + O(\varepsilon^3)$ if and only
if $C$ is diagonal.
\label{lem:kvn-proof-eps}
\end{lemma}
\begin{proof}
We have $M = I + \varepsilon(A-B) + \frac{\varepsilon^2}{2} (A^2 - 2AB +
B^2) + O(\varepsilon^3)$. We also note that $D$ and $D^*$ commute, a fact that we will use repeatedly below. 
We begin by observing that $\Tr(D^*M^*D^*DMD)$ can be written as
\begin{align*}
 \Tr(& D^*(I + \varepsilon(A-B)^*
  + \frac{\varepsilon^2}{2} (A^2 - 2AB + B^2)^*
 + O(\varepsilon^3))D^*
 \\
&  \quad \phantom{} \cdot D(I + \varepsilon(A-B) + \frac{\varepsilon^2}{2} (A^2 - 2AB + B^2)
+ O(\varepsilon^3))D)  \\
 &  =  \Tr((D^*D)^2) + \varepsilon\left[\Tr(D^*(A-B)^*D^*D^2) + \Tr((D^*)^2 D (A-B)D)\right] \\
& \quad \phantom{} +        \frac{\varepsilon^2}{2}\left[\Tr(D^*(A^2-2AB+B^2)^* D^*D^2)
          +
   \Tr((D^*)^2D(A^2-2AB+B^2)D)\right] \\
    & \quad \phantom{} +   \varepsilon^2\left[\Tr(D^*(A-B)^*D^*D(A-B)D\right] + O(\varepsilon^3).
\end{align*}
Let us examine each term. The $\varepsilon$ term is zero, since $(A-B)^*
= -(A-B)$. Furthermore, since $A, B$ are skew-Hermitian, $(AB)^* = B^*A^* =
BA$, so the $\varepsilon^2/2$ term simplifies to
$\varepsilon^2\Tr((D^*D)^2(A-B)^2)$. The combined coefficient of $\varepsilon^2$ is then
\[-\Tr((D^*D)(A-B)^*(A-B)(D^*D)) + \Tr(D^*(A-B)^*D^*D(A-B)D).\]
Now, replacing $D$ with $UD_0U^*$, we
find that this coefficient is equal to
\[-\Tr(UD_0^*D_0U^*(A-B)^*(A-B)UD_0^*D_0U^*) +
  \Tr(UD_0^*U^*(A-B)^*UD_0^*D_0U^*(A-B)UD_0U^*),\]
which simplifies to
\[-\Tr(D_0^*D_0U^*(A-B)^*(A-B)UD_0^*D_0) +
  \Tr(D_0U^*(A-B)^*UD_0^*D_0U^*(A-B)UD_0).\]
For ease of notation, let $C = U^*(A-B)U$ and note that $C$ is skew-Hermitian. Then the expression equals
\begin{align*}
  -\Tr(D_0^*D_0C^*CD_0^*D_0) &  +  \Tr(D_0^*C^*D_0^*D_0CD_0) \\ &=   -\sum_{i,j}|d_i|^4|C[i,j]|^2 + \sum_{i,j}|d_i|^2|d_j|^2|C[i,j]|^2
  \\ & = 
  \sum_{i,j}|d_i|^2(|d_j|^2 - |d_i|^2)|C[i,j]|^2 \\ &= \sum_{i<j} |d_i|^2(|d_j|^2
  - |d_i|^2) |C[i,j]|^2 + \sum_{i<j} |d_j|^2(|d_i|^2 - |d_j|^2) |C[j,i]|^2
  \\ & = 
  \sum_{i<j}(|d_i|^2 -|d_j|^2)(|d_j|^2-|d_i|^2)|C[i,j]|^2 \\& =  -\sum_{i<j}
          (|d_i|^2 - |d_j|^2)^2|C[i,j]|^2,
\end{align*}
as desired.
\end{proof}

In preparation for the next subsection, we now describe some invariant polynomials in $\Inv_{X,Z}$ that we'll use below, where $X = D^{-1}\U_nD$ and $Z = D\U_nD^{-1}$. For an $n \times n$ matrix $M$ and subsets $S,T \subseteq [n]$, we use $M_{S,T}$ to denote the $|S| \times |T|$ sub-matrix of $M$ whose rows are indexed by the elements of $S$ and whose columns are indexed by the elements of $T$.
\begin{lemma}\label{lem:invariantring}
For each $k$, the function $p_k(M)$ defined as
  \[\sum_{\substack{S, T \subseteq [n]\\ |S|=|T|=k}}|\det((DMD)_{S, T})|^2\]
  lies in $\Inv_{X,Z}$, where $X = D^{-1}\U_nD$ and $Z = D\U_nD^{-1}$.
  \label{lem:inv-three-conjugates-unitary}
\end{lemma}
We have been implicitly  using this fact for $p_1$, which happens to be the complex Frobenius norm-squared function. In particular, $\Tr(D^*M^*D^*DMD) = p_1(M)$.

\begin{proof}
We will show that the $\pm p_k(M)$ are the coefficients of the characteristic polynomial of $D^* M^* D^* D M D$, and that this characteristic polynomial is unchanged if we replace $M$ by $xMz$ for $x \in X, z \in Z$. 

We show the latter first. Let $x = D^{-1} M_1 D$ and $z = D M_3 D^{-1}$ with $M_1,M_3 \in \U_n$. Then
\begin{align*}
\det(D^* & (xMz)^* D^* D (xMz) D - \lambda I) \\
& = \det(D^* (D^{-1} M_1 D MD M_3 D^{-1})^* D^* D (D^{-1} M_1 D MD M_3 D^{-1}) D - \lambda I) \\
& = \det(D^* (D^*)^{-1} M_3^* D^* M^* D^* M_1^* (D^*)^{-1} D^* D (D^{-1} M_1 D MD M_3 D^{-1}) D - \lambda I) \\
& = \det(M_3^* D^* M^* D^* D MD M_3 - \lambda I) \qquad \text{(by canceling inverses, using $M_1 \in \U_n$)} \\
& = \det(D^* M^* D^* D MD - \lambda I) \qquad\qquad\,\,\,\,\, \text{(conjugate by $M_3$, using $M_3 \in \U_n$)}.
\end{align*}
Thus each coefficient of this polynomial in
$\lambda$ is an invariant function in $\Inv_{X,Z}$.

Next we show that the coefficients of the above characteristic polynomial are $\pm p_k(M)$. In general, the coefficient of $\lambda^{n-k}$
in the expression $\det(A - \lambda I)$ is \[(-1)^{n-k}\sum_{\substack{S \subseteq [n]\\ |S|
  = k}} \det(A_{S, S}),\] where $A_{S, T}$ is the submatrix indexed by
sets $S, T \subseteq [n]$. So in our case the coefficient of $\lambda^k$ in the characteristic polynomial of $D^*M^*DDMD$ is
\begin{align*}
 (-1)^{n-k}&\sum_{\substack{S \subseteq [n]\\ |S|
      = k}} \det((D^*M^*D^*DMD)_{S, S}) \\& =   (-1)^{n-k}\sum_{\substack{S \subseteq [n]\\ |S|
      = k}} \sum_{\substack{T \subseteq [n]\\ |T| = k}}
      \det((D^*M^*D^*)_{S,T})\det((DMD)_{T, S})   \\ & =  (-1)^{n-k}\sum_{\substack{S,T \subseteq [n]\\
  |S| = |T|=  k}} |\det((DMD)_{S, T})|^2 \\
  & = (-1)^{n-k}p_k(M),
\end{align*}
and this is in $\Inv_{X,Z}$ as claimed. 
\end{proof}
Note that because of the complex norm, these functions are not polynomials in the complex matrix entries (but they are polynomials in the entries of the natural real embedding into twice the dimension). In the next section we restrict the containing group to a special unitary group, which has the crucial side effect of making these functions polynomials in the \emph{complex} matrix entries, because the complex conjugate of a given entry can be found as low-degree polynomial of the entries of the inverse matrix. 

\subsection{Proof of Theorem \ref{thm:construction}}\label{sec:degreeq}
Recall that the containing group of our construction in Theorem \ref{thm:construction} is the following unitary group:
\begin{align*}
G = \SU_{n/2,n/2} = \{M \in \GL_n(\C) \mid M^*QM = Q \text{ and} \det(M)=1\}, \quad \text{with } \ \
Q = \left( \begin{array}{c|c}I &  \\ \hline  & -I \end{array}\right).     
\end{align*}
Our TPP triple in $G$ stated in Theorem \ref{thm:construction} is obtained by taking the ``almost TPP'' triple of Theorem \ref{thm:three-conjugates-unitary} (recall that the construction almost satisfies the TPP except for a ``failure at the diagonal'') and intersecting it with $G$. In this section, we will fix the failure at the diagonal using Lemma~\ref{lem:split}, to get a genuine TPP construction and border-separating polynomials of the optimal degree.

Recall that our construction in $G$ is given by 
\begin{align*}
X &=  \{D^{-1}\exp(\varepsilon A)D \mid A \in S\}, \\
Z &=  \{D  \exp(\varepsilon A)D^{-1} \mid A \in S\}, \\
Y_q  &= \big\{ \exp(\varepsilon A) \mid A\in S, \ \text{entries of  $A$ lie in $\mathcal{I}_q$}
 \big\},
\end{align*}
where $\mathcal{I}_q \subset \C$ is the set of $a+ib$ with $a,b \in [-\lceil\sqrt{q}/2 \rceil, \lceil \sqrt{q}/2\rceil]$ and 
\[S = \left \{\left ( \begin{array}{c|c} A_0& 0 \\ \hline 0 & A_1 \end{array}\right ) \mid A_0, A_1 \text{ are skew-Hermitian with zeros on the diagonal}\right \}.\]

We will prove Theorem \ref{thm:construction} in three steps.

\textbf{Step 1 summary}: We show that the sets $X$ and $Z$ in the theorem statement are subsets of $D^{-1}\U_nD \cap G$ and $D\U_nD^{-1} \cap G$ and hence the ring of invariant polynomials contains the polynomials specified in Lemma~\ref{lem:invariantring}. To do this we'll describe the Lie algebra of $U_n \cap G$, and show that $D$ is in $G$.

\textbf{Step 2 summary}: Next, we describe the separating polynomial for $Y_q$ in this invariant ring, and show that it has degree $O(q)$ as promised. 

\textbf{Step 3 summary}: Finally we describe the functions $f_X$, $f_Z$, $p_X$, and $p_Z$ and apply Lemma~\ref{lem:split}. For this we prove a version of the double product property for $X$ and $Z$.

\medskip 

\textbf{Step 1:}
Our TPP triple is defined relative to $D$, and we now show that the $D$ defined in Section~\ref{sec:construction-details} is indeed in $G$ as claimed. Recall that we use $J$ for the matrix with ones on the antidiagonal $(n+1-i,i)$ and zeroes elsewhere (not the all-ones matrix).
\begin{lemma}
\label{lem:D}
The matrix $D=UD_0U^*$ is an element of $G$, where
\[U = \frac{1}{\sqrt{2}}\cdot \left ( \begin{array}{c|c} I & J \\ \hline J & -I \end{array}\right
  ),\]
  and
  $D_0 = \diag(n, n-1, n-2, \ldots, n/2+1, (n/2+1)^{-1}, (n/2+2)^{-1}, \ldots, n^{-1})$.
\end{lemma}

\begin{proof}
In fact this holds if we replace $D_0$ with any diagonal matrix of the form 
\[D_0 = \left ( \begin{array}{c|c} D_1 & 0 \\ \hline 0 & JD_1^{-1}J \end{array}\right
  ),\]
where $D_1$ is real and invertible. To see why, note that the desired equation $D^*QD = UD_0U^*QUD_0U^* = Q$ is equivalent to $D_0 U^*QU D_0 = U^*QU$, which follows straightforwardly from
\[
U^*QU = \left ( \begin{array}{c|c} 0 & J \\ \hline J & 0 \end{array}\right).
\]
Finally, $\det(D)=1$ since $\det(D_0)=1$ and $U$ is orthogonal.
\end{proof}

Given that we will define our TPP triple by intersecting the three subgroups from the previous subsection with $G$, we start by describing the intersection of $Y = \U_n$ with $G$:
\begin{lemma}
\[
\U_n \cap G = \left\{\left(\begin{array}{c|c} M_{1,1}& 0 \\ \hline 0 & M_{2,2} \end{array}\right) : M_{1,1}, M_{2,2} \in \U_{n/2} \text{ and } \det(M_{1,1} M_{2,2}) = 1\right\}.
\]
\end{lemma}

\begin{proof}
The matrices $M$ in subgroup $\U_n \cap G$ are exactly those with determinant~$1$ that satisfy these equations in terms of their $n/2 \times n/2$ blocks:
\begin{align*}
\left ( \begin{array}{c|c} M_{1,1}^*& M_{2,1}^* \\ \hline M_{1,2}^* & M_{2,2}^* \end{array}\right
  ) \cdot \left ( \begin{array}{c|c} I&  \\ \hline  &
                                                       -I \end{array}\right  ) \cdot  \left ( \begin{array}{c|c} M_{1,1}& M_{1,2} \\ \hline M_{2,1} & M_{2,2} \end{array}\right
  ) & =   \left ( \begin{array}{c|c} I&  \\ \hline  &
                                                       -I \end{array}\right )                                         
  \\
\left ( \begin{array}{c|c}M_{1,1}^*& M_{2,1}^* \\ \hline M_{1,2}^* & M_{2,2}^*  \end{array}\right
  ) \cdot \left ( \begin{array}{c|c} I&  \\ \hline  &
                                                       I \end{array}\right   ) \cdot  \left ( \begin{array}{c|c} M_{1,1}& M_{1,2} \\ \hline M_{2,1} & M_{2,2}  \end{array}\right
  ) & =   \left ( \begin{array}{c|c} I&  \\ \hline  &
                                                       I\end{array}\right ).
\end{align*}
It follows that
\begin{align}
  M_{1,1}^*M_{1,1} - M_{2,1}^*M_{2,1} & =  I \label{eq:first}\\
  M_{1,2}^*M_{1,2} - M_{2,2}^*M_{2,2} & =  -I \label{eq:second}\\
  M_{1,1}^*M_{1,1} + M_{2,1}^*M_{2,1} & =  I \label{eq:third}\\
  M_{1,2}^*M_{1,2} + M_{2,2}^*M_{2,2} & =  I \label{eq:fourth} \\
  M_{1,1}^*M_{1,2} - M_{2,1}^*M_{2,2} & =  0 \label{eq:fifth} \\
  M_{1,1}^*M_{1,2} + M_{2,1}^*M_{2,2} & =  0. \label{eq:sixth}
\end{align}
Adding \eqref{eq:first} and \eqref{eq:third}, we get $M_{1,1}^* M_{1,1} = I$. Similarly, subtracting \eqref{eq:second} from \eqref{eq:fourth}, we get $M_{2,2}^* M_{2,2} = I$. Next, adding \eqref{eq:fifth} and \eqref{eq:sixth} gives $M_{1,1}^* M_{1,2} = 0$, but since we have already shown that $M_{1,1}^*$ is invertible, this implies $M_{1,2} = 0$. Taking the difference of the last two equations gives $M_{2,1}^* M_{2,2} = 0$, and as we have already shown that $M_{2,2}$ is invertible, we get $M_{2,1} = 0$.

Conversely, it is readily verified that any matrices $M_{i,j}$ satisfying $M_{1,1}^*M_{1,1} = I$, $M_{2,2}^*M_{2,2} = I$, $M_{1,2} = M_{2,1} = 0$, and $\det(M_{1,1} M_{2,2})=1$ yield an element of $U_n \cap G$ as above.
\end{proof}
A simple corollary, using the fact that the Lie algebra of the unitary group is the algebra of skew-Hermitian matrices,   is the following.
\begin{corollary}
The Lie algebra of $\U_n \cap G$ is
\[
\Lie(\U_n \cap G) = \left\{ \left ( \begin{array}{c|c} A_0& 0 \\ \hline 0 & A_1 \end{array}\right) 
 \mid A_0^* = -A_0, \ A_1^* = -A_1, \ \Tr(A_0)+\Tr(A_1)=0\right\}.
\]
In particular, for any such matrix $A$ and for all sufficiently small $\varepsilon > 0$, the exponential $\exp(\varepsilon A)$ is in $\U_n \cap G$.
\end{corollary}

Now our subsets $X$ and $Z$ are contained in subgroups $D^{-1}\U_nD$ and $D\U_nD^{-1}$, respectively,  so the invariant polynomials of Lemma \ref{lem:inv-three-conjugates-unitary} are invariant under left multiplication by $X$ and right multiplication by $Z$ as before.

\medskip
\textbf{Step 2:}
A key property of $G$ as the containing group, given that the invariant functions of Lemma \ref{lem:inv-three-conjugates-unitary} depend on the complex conjugate of matrix entries in addition to the matrix entries themselves, is described in the following lemma:
\begin{lemma}
For any $M \in G$,
\[\overline{M[i,j]} = (-1)^{i+j}\det((QMQ)_{-i,-j})\]
where $(QMQ)_{-i,-j}$ denotes the submatrix obtained by deleting the $i$-th row and $j$-th column.
\label{lem:polynomial-complex-conjugate}
\end{lemma}
\begin{proof}
    By the definition of $G$, $M^* = QM^{-1}Q$, and so $\overline{M[i,j]}$ is the $(i,j)$ entry of $QM^{-T}Q$. Plugging in the cofactor formula for the inverse completes the proof (note that $\det(M)^{-1}=1$, so it does not appear above).
\end{proof}

\begin{theorem}
Let $D$ be as in Lemma \ref{lem:D} and $Y_q$ as above.
For $A, B \in Y_q$, if we set $M = \exp(\varepsilon A)\exp(\varepsilon B)^{-1}$, then
\[\frac{\Tr((D^*D)^2) -\Tr(D^*M^*D^*DMD)}{\varepsilon^2} = c + O(\varepsilon),\]
where $2(n!)^2c$ is a nonnegative integer satisfying $c = O(q)$ and for which $c=0$ iff $A = B$.
\end{theorem}
\begin{proof}
Apply Lemma~\ref{lem:kvn-proof-eps} and observe that the norms of the entries of $U^*(A-B)U$ are integers---after clearing denominators by multiplying by $2(n!)^2$---with absolute value at most $O(q)$. The conclusion of that lemma is that $U^*(A-B)U$ is diagonal. When $A, B$ have the block structure of the matrices in $S$, this implies that $A-B$ itself is diagonal, from which we conclude that $A-B = 0$.
\end{proof}
We note that $|Y_q| \ge q^{n^2/4 - n/2}$ and so if we define the polynomial $r(z)$ to be 1
on 0 and 0 on any other positive integer multiple of $1/(2(n!)^2)$ up to $O(q)$,
then the function 
\[p_0(M) = r((\Tr((D^*D)^2)-
\Tr(D^*M^*D^*DMD))/\varepsilon^2)\] is a separating function in the ring of invariant
functions, of the desired degree $O(q)$. The fact that this is a \emph{polynomial}
follows from Lemma \ref{lem:polynomial-complex-conjugate}. Because $\Tr(D^*M^*D^*DMD)$ is the polynomial $p_1(M)$ of Lemma \ref{lem:invariantring}, it is invariant under the left/right action of $D^{-1}U_nD$ and $DU_nD^{-1}$. From step 1 we know that $X \subseteq D^{-1}U_nD$ and $Z \subseteq DU_nD^{-1}$, and hence $p_0 \in \Inv_{X,Z}$.

\medskip 

\textbf{Step 3:}
We now turn to describing the  functions $f_X$, $f_Z$, $p_X$, and $p_Z$, beginning with showing that $X$ and $Z$ satisfy a version of 
 the double product property.
\begin{lemma}
For $x \in D^{-1}SD$ and $y \in DSD^{-1}$,  $x+y=0$ iff $x = y = 0$.
\label{lem:DPP}
\end{lemma}
\begin{proof}
Let $x = D^{-1}AD$ and $y = DBD^{-1}$ for some $A, B \in S$, with $D$ as defined above. If $x+y = 0$ then  $D^{-2}AD^{2} = -B$ and hence is skew-Hermitian, i.e.,
\[-(D^*)^2A^*(D^*)^{-2} = D^{-2}AD^2,\]
which implies $(D^*D)^2A(D^*D)^{-2} = A$ since $D$ commutes with $D^*$ and $A^*=-A$. Then substituting $D = UD_0U^*$ (and using that $U^* = U^{-1}$), we get
\[(D_0^*D_0)^2(U^*AU)(D_0^*D_0)^{-2} = U^*AU.\]
Because $(D_0^*D_0)^2$ is a diagonal matrix with distinct entries, this equality can occur only when $U^*AU$ is diagonal. As noted above, for $A \in S$, $U^*AU$ is diagonal iff $A = 0$. Then $x + y = 0$ implies $B = 0$ as well, and we have $x = y = 0$ as claimed.
\end{proof}

\begin{lemma}
For $X,Z$ as above, there exist functions $f_X, f_Z$ and polynomials $p_X, p_Z$ over $\C$ satisfying the first condition of Lemma~\ref{lem:split} with $d_X = d_Z = n^2/4 - n/2$.
\end{lemma}

\begin{proof}
Consider the natural map $\tau$ from $\C^{d_X} = \C^{d_Z} = \C^{n^2/4-n/2}$ into $S$, which places half the complex coordinates above the diagonal of each block of $S$, and puts their complex conjugates below the diagonal.
Define $\theta \colon \C^{d_X} \times \C^{d_Z} \to \Lie(G)$ by
\[\theta(a, b) := D^{-1}\tau(a)D - D\tau(b)D^{-1}\]
To see that the image of $\theta$ lies in $\Lie(G)$, note that the image of $\tau$ lies in $S$, $DSD^{-1} \subseteq \Lie(X)$, and $D^{-1} S D \subseteq \Lie(Y)$. 
The function $\theta$ is $\R$-linear but not $\C$-linear, because of the use of complex conjugation. Lemma~\ref{lem:DPP} implies that $\theta$ has trivial kernel and therefore has a left inverse $\psi$, i.e., an $\R$-linear transformation $\psi \colon \Lie(G) \to \C^{d_X} \times \C^{d_Z}$ such that $\psi(\theta(a,b)) = (a,b)$ for all $(a,b) \in \C^{d_X} \times \C^{d_Z}$.

Define $f_X \colon \C^{d_X} \to \Lie(X)$ and $f_Z \colon \C^{d_Z} \to \Lie(Z)$ by
\[f_X(a) = D^{-1}\tau(a)D \qquad\text{and}\qquad 
f_Z(b) = D\tau(b)D^{-1}, \]
so that $\theta(a,b) = f_X(a) - f_Z(b)$. Let $\pi_X \colon \C^{d_X} \times \C^{d_Z} \to \C^{d_X}$ be projection to the first coordinate, and let $\pi_Z \colon \C^{d_X} \times \C^{d_Z} \to \C^{d_Z}$ be projection to the second coordinate.
We define $p_X \colon \Lie(X) \to \C^{d_X}$ and $p_Z \colon \Lie(X) \to \C^{d_Z}$ by
\[
p_X(M) = \pi_X(\psi(M)) \qquad\text{and}\qquad p_Z(M) = \pi_Z(\psi(M)).
\]
Then for all $(a,b) \in \C^{d_X} \times \C^{d_Z}$,
\[
p_X(f_X(a)-f_Z(b)) = p_X(\theta(a,b)) = \pi_X(\psi(\theta(a,b))) = \pi_X(a,b) = a,
\]
and similarly $p_Z(f_X(a)-f_Z(b))=b$. Thus, these functions satisfy the first condition of Lemma~\ref{lem:split}.
\end{proof}

This completes the proof of Theorem~\ref{thm:construction}.

\section{Conclusions and open problems} \label{sec:conclusion}
In this paper we gave an extension of the group--theoretic framework of \cite{cu2003} to infinite groups (Theorem \ref{thm:framework}). Within this framework we explored constructions in Lie groups, raising the key question: do there exist three subsets in $\GL_n$ satisfying the TPP, of size at least $q^{n^2/2 - o_n(n)}$, and admitting separating polynomials of degree at most $q^{1 + o_q(1)}$? If the answer is yes, then $\omega = 2$ (Corollary \ref{cor:irrepbound}).

Towards obtaining such a construction, we developed tools using invariant theory and Lie algebras to simplify the task of designing separating polynomials (Subsection \ref{sec:invariant}). We then put these tools to use in Section \ref{sec:construction} to obtain a construction in $\U_n$ satisfying the target degree bound, with sets of size $q^m$ with $m$ approaching half the ambient dimension. 

This raises several directions for future research: 
\begin{itemize}
\item Can one obtain a construction with sets of size $q^m$ for $m$ approaching half the ambient dimension and separating polynomials of degree at most $q^{1 + o_q(q)}$, but in $\GL_n$ rather than $\U_n$?

\item Can one obtain a construction with separating polynomials of degree at most $q^{1 + o_q(q)}$ in $G_n = \GL_n$ or $G_n = \U_n$ with sets of size $q^{\dim G_n/2 - o(n)}$, rather than $q^{\dim G_n/2 - \Theta(n)}$? (In $\GL_n$ this would imply $\omega=2$, towards which getting such a construction in $U_n$ would be an important step.)

\item Our general framework opens up the possibility of using other infinite groups (not necessarily Lie groups), which remains to be explored.

\item Another type of construction that is now possible is one in a single, fixed infinite group (say, $\GL_3$), with growing families of sets $(X_q, Y_q, Z_q)$. In contrast, our current constructions require us to take $n$ growing as well, or more generally, a growing family of containing groups. 

\end{itemize}

\section*{Acknowledgments}

We are grateful to Peter B\"urgisser and Emma Church for useful discussions, and we thank the American Institute for Mathematics for hosting a SQuaRE, during which initial parts of this work were developed. For other funding information see the title page.

\newcommand{\etalchar}[1]{$^{#1}$}

\appendix

\section{Deferred proofs}
\label{appendix}

\subsection{Proof of Theorem~\ref{thm:framework-border}} \label{sec:proof:border}
\begin{proof}[Proof of Theorem~\ref{thm:framework-border}]
Let $X,Y,Z,R_{\textup{sep}}$ be as in the statement, and let $\{f_{x,z} \mid x \in X, z \in Z\}$ be the claimed set of border-separating functions contained in $\Fun_{\textup{fam}}(R_{\textup{sep}})$. Since $X,Y,Z$ and the set of separating functions are all finite sets, we may assume the 1-parameter families involved have a common domain of definition $(0,\alpha)$, by taking the intersection of the domains of definition of the finitely many 1-parameter families involved.

As in the proof of Theorem~\ref{thm:framework}, for any function $f \colon G \to V$ where $V$ is a complex vector space, we use $\overline{f} \colon \C[G] \to V$ denote its unique linear extension to the group ring $\overline{f}(\sum \alpha_g g) := \sum \alpha_g f(g)$; we will use this both in the case where $V = \C$ and where $V$ is the space of matrices that are the codomain of a representation. 

For any $\varepsilon$ in its domain of definition, let $f_{x,z,\varepsilon}(g) := f_{x,z}(g,\varepsilon)$. Applying $\overline{f_{x,z,\varepsilon}}$ to \eqref{eqn:group-algebra-border} gives
\begin{align} \label{eqn:group-algebra-fxz-border}
\overline{f_{x,z,\varepsilon}}(\overline{A}(\varepsilon) \cdot \overline{B}(\varepsilon)) & = \sum_{x' \in X, z' \in Z} (AB)[x',z'] f_{x,z,\varepsilon}(x'(\varepsilon)z'(\varepsilon)^{-1}) + \overline{f_{x,z,\varepsilon}}(E(\varepsilon)) \\
& = (AB)[x,z](1 + O(\varepsilon)) + O(\varepsilon) = (AB)[x,z] + O(\varepsilon) \nonumber
\end{align}

Since $f_{x,z}$ is in $\Fun_{\textup{fam}}(R_{\textup{sep}})$ by assumption, we can write $f_{x,z}$ as a $\C^{(0, \alpha)}$-linear combination of the functions $\{\rho_{i,j} \mid \rho \in R_{\textup{sep}}, i,j \in [\dim \rho]\}$, say $f_{x,z} = \sum_{\rho \in R_{\textup{sep}}} \sum_{i,j \in [\dim \rho]} M_{x,z,i,j}(\varepsilon) \rho_{i,j}$. Then we define $\widehat{f_{x,z,\varepsilon}}(\rho)$ to be $M_{x,z,*,*}(\varepsilon)$; that is, we have
\[
f_{x,z,\varepsilon}(g) = \sum_{\rho \in R_{\textup{sep}}} \sum_{i,j \in [\dim \rho]} \widehat{f_{x,z,\varepsilon}}(\rho)_{i,j} \rho_{i,j}(g).
\]
for all $x \in X, z \in Z, g \in G$. Finally, extending linearly and applying $\overline{f_{x,z,\varepsilon}}$ to $\overline{A}(\varepsilon) \cdot \overline{B}(\varepsilon)$ as in \eqref{eqn:group-algebra-fxz-border}, we get
\begin{align*}
(AB)[x,z] + O(\varepsilon) & = \overline{f_{x,z,\varepsilon}}(\overline{A}(\varepsilon) \cdot \overline{B}(\varepsilon)) \\
& = \sum_{\rho \in R_{\textup{sep}}} \sum_{i,j \in [\dim \rho]} \widehat{f_{x,z,\varepsilon}}(\rho)_{i,j} \overline{\rho_{i,j}}(\overline{A}(\varepsilon) \cdot \overline{B}(\varepsilon)) \\
& = \sum_{\rho \in R_{\textup{sep}}} \langle \widehat{f_{x,z,\varepsilon}}(\rho), \overline{\rho}(\overline{A}(\varepsilon) \cdot \overline{B}(\varepsilon)) \rangle \\
& = \sum_{\rho \in R_{\textup{sep}}} \langle \widehat{f_{x,z,\varepsilon}}(\rho), \overline{\rho}(\overline{A}(\varepsilon)) \cdot \overline{\rho}(\overline{B}(\varepsilon)) \rangle.
\end{align*}
For each $\varepsilon \in (0,\alpha)$, the summation and inner product are $\C$-linear functions whose coefficients are independent of the input matrices $A,B$, so they are ``free'' in a bilinear algorithm. 

The product $\overline{\rho}(\overline{A}(\varepsilon)) \cdot \overline{\rho}(\overline{B}(\varepsilon))$ is a product of $d_\rho \times d_\rho$ matrices, where $d_\rho = \dim \rho$. Taking the limit as $\varepsilon \to 0$, the tensor rank of the preceding expression thus gives us a bound on the \emph{border} rank of the matrix product $AB$ (using $\overline{\rk}$ to denote border rank):
\[
\overline{\rk} \langle |X|, |Y|, |Z| \rangle \leq \sum_{\rho \in R_{\textup{sep}}} \rk \langle d_\rho, d_\rho, d_\rho \rangle.
\]
Exactly as in the finite group case \cite{cu2003}, by symmetrizing we effectively get a square matrix multiplication of size $(|X|\,|Y|\,|Z|)^{1/3}$ on the left side, and by the tensor power trick the right side here can be replaced by $\sum_{\rho \in R_{\textup{sep}}} (\dim \rho)^{\omega}$. Since the value of $\omega$ is the same whether calculated using rank or border rank, the left side becomes $(|X|\,|Y|\,|Z|)^{\omega/3}$. 
\end{proof}

\subsection{Proof of Lemma~\ref{lem:maxdim}} \label{sec:proof:maxdim}
\begin{proof}[Proof of Lemma \ref{lem:maxdim}]
Let $\lambda = (\lambda_1, \lambda_2, \ldots, \lambda_n)$ be a partition of $s$, and let $V_\lambda \in \Irr_s(\GL_n)$ be the corresponding representation. By \cite[Equation 15.17]{fulton1991representation},
\[\dim V_\lambda = \prod_{1 \le i < j \le n}
  \frac{\lambda_i-\lambda_j+ j-i}{j-i},\]
and by the AM-GM inequality,
\[\dim V_\lambda \le \left (\frac{\sum_{1 \le i < j \le n}
  \frac{\lambda_i-\lambda_j + j-i}{j-i}}{\binom{n}{2}}
\right )^{\binom{n}{2}}.\]
The coefficient of $\lambda_i$ in the numerator is
$
\sum_{k=1}^{n-i}\frac{1}{k} - \sum_{k=1}^{i-1}\frac{1}{k},
$
which is at most $\sum_{k=1}^{n-1}1/k$. Thus,
\begin{equation}
\sum_{1 \le i < j \le n}
  \frac{\lambda_i-\lambda_j + j-i}{j-i}
\leq \binom{n}{2} + s \left(\sum_{k=1}^{n-1} 1/k\right), \label{eq:irrepbound}
\end{equation}
because $\sum_i \lambda_i = s$.

When $n=3$, equation \eqref{eq:irrepbound} becomes $\binom{n}{2} + s(1 + 1/2)$, and then we get the bound
\[
\dim V_{\lambda} \leq \left( 1 + \frac{s(1 + 1/2)}{\binom{n}{2}}\right)^{\binom{n}{2}}.
\]
As $1 + s/2 < s$ for all $s \geq 2$, we get the bound $s^{\binom{n}{2}}$ when $n=3$.

When $n > 3$, we instead bound \eqref{eq:irrepbound} using the harmonic series estimate that $\sum_{k=1}^{n-1} 1/k \leq \ln(n) + \gamma$, where $\gamma$ is Euler's constant. Plugging this back into our bound for $\dim V_{\lambda}$ yields
\[ \dim V_\lambda \le \left (1 + \frac{s (\ln
      (n)+\gamma)}{\binom{n}{2}} \right )^{\binom{n}{2}}.\]
For $n \geq 4$ and $s \ge 2$, we have $1+s(\ln(n) + \gamma) / \binom{n}{2} < s$, and so we again get the claimed bound of $s^{\binom{n}{2}}$.
\end{proof}

Even for $n=2$, we also get a bound of $(1+s)^{\binom{n}{2}}$, which is good enough for use in Corollary~\ref{cor:irrepbound}.

\begin{remark}
    \label{rem:lower-bound-on-gln-irrep}
    The formula at the beginning of the proof of Lemma~\ref{lem:maxdim} is exact, and we can use the exact formula to show that our upper bound is, in various regimes, asymptotically almost tight. Let $s = d\binom{n}{2}$, and let $\lambda_i = d(n-i)$ for $i=1,\dotsc,n$. Using the exact formula, we find that
    \begin{align*}
    \dim V_{\lambda} & = \prod_{1 \leq i < j \leq n} \frac{d(n-i) - d(n-j) + j-i}{j-i} \\
    & = \prod_{1 \leq i < j \leq n} (d+1) = (d+1)^{\binom{n}{2}}.
    \end{align*}
    This is the same as $(1+ s/\binom{n}{2})^{\binom{n}{2}}$. When $n = O(1)$ and $s \to \infty$, this becomes $\Theta(s)^{\binom{n}{2}}$. Note that the proof of Lemma~\ref{lem:maxdim} actually yields a bound of $(1+s (\ln (n) + \gamma)/\binom{n}{2})^{\binom{n}{2}}$; while we do not need this tighter bound for its application in Corollary~\ref{cor:irrepbound}, we remark that the above ``staircase'' representation is close to showing this bound is tight.
\end{remark}

\subsection{Representations of $\SU_{n/2,n/2}$} \label{sec:suvssl}

The fact that the irreducible representations of $\SU_{n/2,n/2}$ are the restrictions of those of $\SL_n(\C)$ is a standard fact, amounting to the assertion that the Lie algebra of $\SL_n(\C)$ is the complexification of that of $\SU_{n/2,n/2}$. Here we briefly review this theory. We will focus on the case of polynomial representations, because that is the case of greatest interest to us, but in fact all the irreducible representations of these groups are polynomial representations.

It is not hard to check that $\Lie(\SU_{n/2,n/2})$ consists of matrices with block decompositions
\[
\left( \begin{array}{c|c}A & B \\ \hline C & D \end{array}\right)
\]
into $(n/2) \times (n/2)$ submatrices satisfying $A^*=-A$, $B^* = C$, $D^*=-D$, and $\Tr(A) + \Tr(D) = 0$, while $\Lie(\SL_n(\C))$ consists of all matrices with trace~$0$. In particular, $\Lie(\SU_{n/2,n/2})$ is a real subspace of 
$\Lie(\SL_n(\C))$, with 
\[
\dim_\R \Lie(\SU_{n/2,n/2}) = n^2-1 = \dim_\C \Lie(\SL_n(\C)).
\]
Furthermore, it follows from the characterization given above that
\[
\Lie(\SU_{n/2,n/2}) \cap i \Lie(\SU_{n/2,n/2}) = \{0\}.
\]
Thus, by dimension counting $\Lie(\SU_{n/2,n/2})$ spans $\Lie(\SL_n(\C))$ over $\C$.

It follows that $\SU_{n/2,n/2}$ is dense in $\SL_n(\C)$ under the Zariski topology, in which the closed sets are the algebraic sets (those defined by polynomial equations). To see why, first note that its Zariski closure must be a subgroup of $\SL_n(\C)$ (see, for example, \cite[Proposition~I.1.3]{borel}), and in particular a linear algebraic group because it is an algebraic variety. It is therefore smooth, and so it must be a complex Lie group. Its Lie algebra is a complex vector space contained in $\Lie(\SL_n(\C))$ and containing $\Lie(\SU_{n/2,n/2})$, and so it must be $\Lie(\SL_n(\C))$. Thus, the Zariski closure must be $\SL_n(\C)$.

Now suppose $V$ is any irreducible polynomial representation of $\SL_n(\C)$. If $V$ had an invariant subspace when restricted to $\SU_{n/2,n/2}$, then the invariance of that subspace would be equivalent to certain polynomials vanishing on $\SU_{n/2,n/2}$, which would then have to vanish on its Zariski closure $\SL_n(\C)$ as well. 

To see the polynomial conditions, suppose $v_1, \dotsc, v_m$ is a basis for an $\SU_{n/2,n/2}$-invariant subspace of $V$, with $d = \dim V$ and $m<d$. For each $v_i$, the condition that $g \cdot v_i$ lies in the span of $\{v_1, \dotsc, v_m\}$ for all $g \in \SU_{n/2,n/2}$ can be rephrased in terms of vanishing polynomials as follows. Consider the $d \times (m+1)$ matrix whose first $m$ columns are $v_1, \dotsc, v_m$, and whose last column is $g \cdot v_i$, where the latter is expressed as polynomials in the entries of $g \in \SU_{n/2,n/2}$. Then the fact that $g \cdot v_i$ is in the span of $\{v_1, \dotsc, v_m\}$ is equivalent to this matrix having rank $m$, which is equivalent to all of its $(m+1) \times (m+1)$ minors vanishing. This gives us a set of polynomials for each $i=1,\dotsc,m$, and $\{v_1, \dotsc, v_m\}$ is an invariant subspace if and only if every polynomial in the union of these sets vanishes on all of $\SU_{n/2,n/2}$. 

Thus, the restriction of an irreducible representation of $\SL_n(\C)$ to $\SU_{n/2,n/2}$ must remain irreducible. 

Conversely, every polynomial representation of $\SU_{n/2,n/2}$ extends to its Zariski closure, because the polynomial identities required for the representation to work still hold on the closure.
\end{document}